\documentclass{article}
\usepackage{amsmath}
\usepackage{amsfonts}
\usepackage{amssymb}
\usepackage{amsthm}
\usepackage{mathtools}
\usepackage{mathrsfs}
\usepackage{tikz,ytableau}
\usetikzlibrary{arrows}
\usepackage[letterpaper,top=2cm,bottom=2cm,left=3cm,right=3cm,marginparwidth=1.75cm]{geometry}

\usepackage[colorlinks=true, allcolors=blue]{hyperref}
\usepackage[font=footnotesize,labelfont=bf]{caption}
\usepackage{bm}
\usepackage[noadjust]{cite}

\usepackage{pgfplots}
\usetikzlibrary{arrows.meta,calc}
\pgfplotsset{compat=1.18}

\usepackage{siunitx}

\newcommand{\dashentry}{\phantom{\num{6720}}\llap{--}}

\newtheorem{theorem}{Theorem}[section]
\newtheorem{corollary}[theorem]{Corollary}
\newtheorem{lemma}[theorem]{Lemma}
\newtheorem{proposition}[theorem]{Proposition}
\newtheorem{problem}[theorem]{Problem}

\theoremstyle{remark}
\newtheorem{definition}[theorem]{Definition}
\newtheorem{remark}[theorem]{Remark}
\newtheorem{example}[theorem]{Example}

\title{Conference signals: Applications, existence, and constructions}
\author{
Kean Fallon\thanks{Department of Mathematics, Iowa State University, Ames, IA}
\and
Joseph W.\ Iverson\footnotemark[1]}
\date{}

\begin{document}

\maketitle
\begin{abstract}
A conference signal is a complex-valued function on a finite abelian group that vanishes at 0 in both the time and frequency domains, and is otherwise flat in both domains. 
For example, the Legendre symbol is a conference signal on $\mathbb{Z}_p$ since it vanishes at $0$, takes values $\pm 1$ away from zero, and is a scalar multiple of its Fourier transform. 
More generally, any nontrivial multiplicative character on a finite field is a conference signal on its additive group.

Until now, conference signals have mainly been studied by number theorists, who wanted to know if multiplicative characters were the only conference signals on finite fields. 
The state of the art uses fancy tools from symplectic geometry to prove that non-character conference signals exist on non-small fields. 
The same tools show conference signals exist on every abelian group of odd order, without saying what they are.

In this (human-friendly!) paper, we explain how conference signals produce other objects of interest in algebraic combinatorics and applied harmonic analysis, and then we present some new, explicit constructions of conference signals on $\mathbb{Z}_n$.
In particular, we give the first explicit conference signals for composite~$n$ and the first explicit non-character examples for prime~$n$.
\end{abstract}

\section{Introduction}

We study the following.\footnote{This paper was written by humans, for humans.
Chatbots are welcome, too---we just don't know what they like.
Section~\ref{sec: AI} details the ways we did and did not use AI.
}

\begin{definition}
A \textbf{conference signal} of order $n$ is a function $f \colon A \to \mathbb{C}$ on an abelian group with order~$n$ such that:
    \begin{itemize}
    \item[(i)]
    $f(0) = 0$,
    \item[(ii)]
    $|f(t)| = 1$ for every $t \neq 0$ in $A$,
    \item[(iii)]
    $\hat{f}(0) = 0$,
    \item[(iv)]
    $|\hat{f}(\alpha)| = \sqrt{n}$ for every $\alpha \neq 0$ in $\hat{A}$.
    \end{itemize}
A conference signal is \textbf{normalized} if $f(1) = 1$ for some distinguished element $1 \neq 0$ in $A$.
\end{definition}

Here, $\hat{f} \colon \hat{A} \to \mathbb{C}$ is the (unnormalized) Fourier transform with
\[
\hat{f}(\alpha)
= \sum_{t \in A} f(t) \overline{\langle t, \alpha \rangle}
\qquad
\forall
\alpha \in \hat{A},
\]
where $\hat{A}$ is the Pontryagin dual (character) group of $A$, written additively.
For the special case $A = \mathbb{Z}_n:=\{0,\ldots,n-1\}$ with addition mod~$n$, we identify $\hat{A}$ with $\mathbb{Z}_n$ in the usual way by setting $\langle t,s \rangle := \omega_n^{st}$ for $\omega_n := e^{2\pi \mathrm{i}/n}$.

In~\cite{Benoist:24}, conference signals are called ``unimodular $\mathcal{C}$-functions''.
We call them ``conference signals'' because they are closely related to (complex) conference matrices and this connection drives a lot of their applications; see Section~\ref{sec: applications}.
Just as conference matrices are reminiscent of Hadamard matrices, conference signals are reminiscent of \textit{CAZACs}~\cite{Chu:72,Bjorck:90,BenedettoCM:19}, i.e., signals that are completely flat in both the time and frequency domains (without a hole at~$0$).

\begin{example}
The vector $\left[ \begin{smallmatrix*}[r] 0 \\ 1 \\ -1 \end{smallmatrix*} \right]$ defines a conference signal on $\mathbb{Z}_3$ since its Fourier transform is $\left[ \begin{smallmatrix*}[r] 0 \\ -\mathrm{i}\sqrt{3} \\ \mathrm{i}\sqrt{3} \end{smallmatrix*} \right]$.
More generally, for any odd prime~$p$, the \textbf{Legendre symbol} $\chi_p \colon \mathbb{Z}_p \to \mathbb{C}$ is given by $\chi_p(0) = 0$, $\chi_p(t) = 1$ when $t$ is a quadratic residue (nonzero square) in $\mathbb{Z}_p$, and $\chi_p(t) = -1$ otherwise.
A famous result of Gauss~\cite{Gauss:11,BerndtEW:98} says that 
\begin{equation}
\label{eq: Legendre Fourier}
\widehat{\chi}_p = 
\begin{cases}
\sqrt{p} \chi_p & \text{if }p = 1 \bmod 4, \\
-\mathrm{i} \sqrt{p} \chi_p & \text{if }p = 3 \bmod 4,
\end{cases}
\end{equation}
so $\chi_p$ is a conference signal.
Even more generally, a \textbf{multiplicative character} on a finite field $\mathbb{F}_q$ is a function $f \colon \mathbb{F}_q \to \mathbb{C}$ such that $f(0) = 0$, $f(1) = 1$, and $f(st) = f(s)f(t)$ for all $s,t \in \mathbb{F}_q$; it is \textbf{nontrivial} if $f(t) \neq 1$ for some $t \neq 0$.
After an appropriate identification of $\widehat{\mathbb{F}_q}$ with $\mathbb{F}_q$, the theory of Gauss sums gives that every nontrivial multiplicative character $f$ is a \textbf{con-eigenvector} of the Fourier transform, meaning that $\hat{f}$ is a scalar multiple of $\overline{f}$ (the entrywise complex conjugate).
It follows that every nontrivial multiplicative character on a finite field is a conference signal.
\end{example}

Before now, conference signals have mostly been studied by number theorists interested in the following problem of Harvey Cohn.

\begin{problem}[Harvey Cohn 1994, p202 of~\cite{Montgomery:94}]
Is every conference signal over a finite field a scalar multiple of a multiplicative character?
\end{problem}

The answer is ``no'' for every field of order~$8$ or more.
For non-prime orders, Choi and Siu~\cite{ChoiS:00} found a continuum of normalized conference signals; their proof is constructive.
In the prime-order case, Bir{\'o}~\cite{Biro:99} instead showed there are only finitely many normalized conference signals. Cohn's problem remained open for prime fields for 30~years until its recent (human!) solution by Benoist.

\begin{proposition}[Yves Benoist 2024, \cite{Benoist:24}]
\label{prop: Benoist}
~
\begin{itemize}
\item[(a)]
If $n$ is odd, then $\mathbb{Z}_n$ admits an odd conference signal.
\item[(b)]
If $p\geq 11$ is prime, then $\mathbb{Z}_p$ admits an odd conference signal that is not a scalar multiple of a multiplicative character.
\end{itemize}
\end{proposition}

The proof is non-constructive and ultimately relies on sophisticated machinery from symplectic geometry and Floer homology; see Section~\ref{subsec: Benoist}.
Here, we address the natural followup question:
\begin{center}
\textit{What are those conference signals?}
\end{center}
This question is especially compelling for the authors, who got interested in conference signals because they are useful objects for algebraic combinatorics and applied harmonic analysis; we independently rediscovered the notion without any idea that number theorists had been toying with them for decades.

The next section details applications.
Taking these as motivation, Section~\ref{sec: numerical} reports numerical observations about conference signals, including further conjectures about their (non)existence.
In Section~\ref{sec: existence} we give some nonconstructive proofs of existence using techniques like in~\cite{Benoist:24}.
The first of our main results appears in Section~\ref{sec: explicit composite}, where we construct an explicit infinite family of conference signals with composite order, namely, $n=3p$ for any prime $p>3$.
When $p=3 \bmod 4$, this leads to an \textit{equi-isoclinic tight fusion frame} of $3p$ planes in $\mathbb{C}^{3p-1}$, which was not previously known to exist.
Our second main result appears in Section~\ref{sec: explicit prime}, which gives the first explicit counterexample to Cohn's problem over a prime-order field; the construction works for any prime $p>7$ with $p = 1 \bmod 6$.
Our technique relies on a surprising feature of the discrete Fourier transform over $\mathbb{Z}_p$: when restricted to the $\lambda$-eigenspace of any dilation with $\lambda \neq 1$, the matrix of the Fourier transform is circulant for a certain choice of bases.
The paper ends with discussion in Section~\ref{sec: discussion}.

\section{Applications}
\label{sec: applications}

As we now explain, conference signals produce other objects of interest in algebraic combinatorics and applied harmonic analysis.
For background on harmonic analysis over finite abelian groups, see~\cite{Terras:99}.

\subsection{Conference matrices}

A \textbf{complex conference matrix} of order~$m$ is a matrix $K \in \mathbb{C}^{X \times X}$ such that $|X|=m$, $KK^* = (m-1)I$, and $|K_{ij}| = 1-\delta_{i,j}$ for every $i,j \in X$, where $\delta_{i,j}$ is Kronecker's delta symbol.
To understand the connection with conference signals, let $A$ be a finite abelian group, and associate each $f \colon A \to \mathbb{C}$ with the $A$-circulant matrix $C_f \in \mathbb{C}^{A \times A}$ with entries $(C_f)_{s,t} = f(s-t)$.
Its conjugate transpose is $C_f^* = C_{f^*}$, where $f^*$ is the conjugate reversal with $f^*(t) = \overline{f(-t)}$.
Furthermore, $C_f C_g = C_{f * g}$, where $f*g$ is the convolution with $(f*g)(t) = \sum_{s\in A} f(s) g(t-s)$.
In the following, we also write $1$ for a function that is constantly 1, we write $\delta_0$ for a point mass at~$0$, and we write $\mathbf{1}$ for a column vector of all-ones.

\begin{theorem}
\label{thm: conference matrices}
Let $A$ be a finite abelian group of order~$n$.
Given $f \colon A \to \mathbb{C}$, the following are equivalent:
    \begin{itemize}
    \item[(i)]
    $f$ is a conference signal,
    \item[(ii)]
    $|f|^2 = 1 - \delta_0$ and $f * f^* = n \delta_0 - 1$,
    \item[(iii)]
    for either choice of $\varepsilon \in \{ \pm 1 \}$,
    \[
    K:= \left[ \begin{array}{cc} 0 & \mathbf{1}^\top \\ \varepsilon \mathbf{1} & C_f \end{array} \right]
    \]
    is a complex conference matrix of order $n+1$.
    \end{itemize}
Furthermore, $K$ is symmetric if and only if $f$ is even and $\varepsilon = 1$, while $K$ is skew-symmetric if and only if $f$ is odd and $\varepsilon = -1$.
\end{theorem}

The equivalence of (i) and (ii) is Lemma~2.4 of~\cite{Benoist:24}.
Here, our interest is the equivalence of (i) and (iii).
The proof is easy, but the result is important.

\begin{proof}[Proof of Theorem~\ref{thm: conference matrices}]
(i)$\iff$(ii) because $f$ is a conference signal if and only if $|f|^2 = 1 -\delta_0$ and $|\hat{f}|^2 = n(1-\delta_0)$, where $|\hat{f}|^2 = \hat{f} \cdot \overline{\hat{f}} = \widehat{f * f^*}$ and $n(1-\delta_0)=(n\delta_0-1)^{\widehat{}}$.
To see that (ii)$\iff$(iii), first compute
\[
KK^*
= \left[ \begin{array}{cc}
\mathbf{1}^\top \mathbf{1} & (C_f \mathbf{1})^* \\
C_f \mathbf{1} & \mathbf{1} \mathbf{1}^\top + C_{f*f^*}
\end{array} \right]
= \left[ \begin{array}{cc}
n & \overline{\hat{f}(0)} \cdot\mathbf{1}^\top \\
\hat{f}(0) \cdot\mathbf{1} & C_{1 + f*f^*}
\end{array} \right].
\]
If (ii) holds, then $C_{1 + f*f^*} = C_{n\delta_0} = nI$ and $\hat{f}(0) = 0$ since (ii) implies (i); furthermore, the diagonal of $C_f$ vanishes and its off-diagonal entries are unimodular since $|f|^2 = 1 - \delta_0$.
Thus, (ii) implies (iii).
Conversely, if (iii) holds, then $C_{1+f*f^*} = n I = C_{n\delta_0}$ so $f*f^* = n \delta_0 - 1$; furthermore, the diagonal of $C_f$ vanishes and its off-diagonal entries are unimodular, so $|f|^2=1-\delta_0$.
Thus, (iii) implies (ii).
The ``furthermore'' part is trivial to check.
\end{proof}

\begin{example}
Taking $f$ to be the Legendre symbol in Theorem~\ref{thm: conference matrices} recovers Paley's construction of conference matrices~\cite{Paley:33,DelsarteGS:71}.
More generally, taking $f$ to be a nontrivial multiplicative character on a finite field in Theorem~\ref{thm: conference matrices} recovers the construction of complex conference matrices in Theorem~5.2 of~\cite{FickusIJM:23}; see also~\cite{EtTaoui:23}.
\end{example}

Complex conference matrices with circulant cores also appear in~\cite{Dita:09}.
Complex conference matrices that are themselves circulant (not their cores) appear in~\cite{EtTaoui:18,FickusS:20}; see also~\cite{UzcateguiContrerasGTV:21}.

\begin{remark}
As noted in~\cite{Dita:09}, for any complex conference matrix $K$ of order~$m$, $H:=\left[ \begin{smallmatrix} K+I & K^*-I \\ K-I & -K^*-I \end{smallmatrix} \right]$ is a \textbf{complex Hadamard matrix} of order $2m$, meaning all its entries are unimodular and $HH^*=2mI$.
Thus, every conference signal of order $n$ gives a complex Hadamard matrix of order $2n+2$.
\end{remark}

\begin{remark}
\label{rem: real CS}
Real conference signals have a combinatorial interpretation, where a function $f \colon A \to \mathbb{R}$ satisfying $|f|^2 = 1 - \delta_0$ is completely determined by the set $D:=\{t \in A : f(t) = -1\}$.
The remaining condition $f*f^*=n\delta_0 - 1$ from Theorem~\ref{thm: conference matrices}(ii) says something about $D$, depending on the value of $n:=|A|$.
Here, $n$ has to be odd for $f$ to be a conference signal since $\hat{f}(0) = \sum_{t\in A} f(t)$ has to vanish.
When $n = 3 \bmod 4$, the autocorrelation condition $f*f^* = n\delta_0 - 1$ holds if and only if $D$ is a \textit{skew Hadamard difference set}~\cite{SeberryY:20}; in this case $f$ is necessarily odd.
Such $D$ exists for some $A$ if and only if $n \equiv 3 \bmod 4$ is a prime power~\cite{CamionM:72}.
When $n = 1 \bmod 4$, the autocorrelation condition $f*f^* = n\delta_0 - 1$ holds if and only if $D$ is a \textit{regular Paley-type partial difference set}~\cite{Ma:94}; in this case $f$ is necessarily even.
Such $D$ exists for some $A$ if and only if~$n \equiv 1 \bmod 4$ is a prime power or else $n \in \{m^4,9m^4\}$ for some odd $m > 1$~\cite{Polhill:10,Wang:20}.
Combining these cases, a real conference signal of order~$n$ exists if and only if $n$ is an odd prime power or $n \in \{m^4,9m^4\}$ for some odd $m > 1$.
\end{remark}

\subsection{Optimal subspace packings}

Our next application comes from the theory of subspace (Grassmannian) packings.
An \textbf{equi-isoclinic tight fusion frame}, abbreviated $\operatorname{EITFF}(d,r,n)$, is a sequence of subspaces in $\mathbb{R}^d$ or $\mathbb{C}^d$ whose orthogonal projections $P_1,\ldots,P_n$ all have rank~$r$ and satisfy $\sum_{i=1}^n P_i = a I$ and $P_i P_j P_i = b P_i$ whenever $i \neq j$, where $a$ and $b$ are constants.
EITFFs are optimal packings in Grassmannian space with respect to the spectral and chordal distances~\cite{DhillonHST:08,FickusJMW:17}, and they find further applications in compressive sensing~\cite{CalderbankTX:15,EldarKB:10,EldarM:09}.

Any symmetric or skew-symmetric complex conference matrix produces an (explicit) EITFF, as detailed in Theorem~5.1 of~\cite{FickusIJM:23}.
Combining that result with Theorem~\ref{thm: conference matrices} gives:

\begin{corollary}
\label{cor: EITFF}
~
\begin{itemize}
\item[(a)]
Every odd conference signal of order~$n \geq 3$ produces an $\operatorname{EITFF}(n-1,2,n)$.
\item[(b)]
Every even conference signal of order~$n\geq 3$ produces an $\operatorname{EITFF}(n,2,n)$.
\end{itemize}
Furthermore, in both (a) and (b), the EITFF is real if the conference signal is real.
\end{corollary}

Likewise, combining Theorem~3 of~\cite{BlokhuisBE:18} with Theorem~\ref{thm: conference matrices} gives:

\begin{corollary}
Every even conference signal of order $n$ produces a real $\operatorname{EITFF}(n+1,2,n+1)$.
\end{corollary}

Applying Proposition~\ref{prop: Benoist}(a) to Corollary~\ref{cor: EITFF}(a), we deduce:

\begin{corollary}
For every odd $n \geq 3$, there exists an $\operatorname{EITFF}(n-1,2,n)$.
\end{corollary}

Sadly, this proof is nonconstructive, underlining the need for explicit conference signals.
Theorem~\ref{thm: n = 3p} below constructs an odd conference signal when $n=3p$ for a prime $p>3$ with $p=3 \bmod 4$.

\subsection{Symplectic equiangular tight frames}
Let $[ \cdot, \cdot ]$ be a symplectic (i.e., alternating and nondegenerate) form on $\mathbb{F}^d$, where $\mathbb{F} \in \{ \mathbb{R}, \mathbb{C} \}$ and $d$ is even.
A \textbf{symplectic equiangular tight frame}, abbreviated $\operatorname{SETF}(d,n)$, is a sequence $\{\varphi_i\}_{i\in X}$ of length $|X|=n$ in $\mathbb{F}^d$ whose \textbf{Gram matrix} $\bigl[ \, [\varphi_i, \varphi_j] \, \bigr]_{i,j\in X} \in \mathbb{F}^{X \times X}$
has exactly one nonzero singular value, repeated with multiplicity~$d$, and satisfies $\bigl| [ \varphi_i, \varphi_j ] \bigr| = 1$ whenever $i \neq j$.
See~\cite{FallonLAA:26,Fallon:26}.
The following is an easy consequence of~\cite[Proposition~3.5]{Fallon:26}.

\begin{theorem}
\label{thm: SETFs}
For any $f \colon A \to \mathbb{F}$, $f$ is an odd conference signal of order $n \geq 3$ if and only if $C_f$ is the Gram matrix of an $\operatorname{SETF}(n-1,n)$.
\end{theorem}

\begin{remark}
When $\mathbb{C}^d$ is instead equipped with its usual (Hermitian) inner product, it is well known that \textit{equiangular tight frames} with circulant Gram matrices are precisely those given by \textit{difference sets} in finite abelian groups.
(This can be deduced, for instance, by combining Theorems~10.3, 11.1, and~12.3 in~\cite{Waldron:18}.)
In this sense, Theorem~\ref{thm: SETFs} suggests an analogy between conference signals and difference sets; Remark~\ref{rem: real CS} gives other evidence of such a link.
\end{remark}

Benoist's existence result implies the following.

\begin{corollary}
For every odd~$n \geq 3$, there exists a complex $\operatorname{SETF}(n-1,n)$.
\end{corollary}

\section{Numerical observations}
\label{sec: numerical}

Motivated by the above applications, the authors set out to perform a numerical search for conference signals.
It didn't take long for us to notice that conference signals come in multi-packs, and furthermore, every one we found exhibited symmetries.
Specifically, the set of conference signals in $\mathbb{C}^{\mathbb{Z}_n}$ is invariant under all of the following (isometric and real-linear) transformations:
\begin{itemize}
\item[(i)]
$cI$, for any unimodular $c \in \mathbb{C}$;
\item[(ii)]
dilation $(D_m f)(t) = f(m^{-1} t)$,
for any unit $m\in \mathbb{Z}_n^\times$;
\item[(iii)]
the \textbf{discrete Fourier transform} (DFT) $\mathcal{F} f = \tfrac{1}{\sqrt{n}} \hat{f}$;
\item[(iv)]
entrywise complex conjugation $Cf = \overline{f}$.
\end{itemize}
Indeed, (i) is clear,
(ii) works since $(D_m f)^{\widehat{}} = D_{m^{-1}} \hat{f}$,
(iii) works since $\hat{\hat{f}}(t)=n f(-t)$,
and (iv) works since $\widehat{\overline{f}} = \hat{f}^*$.
These transformations generate a group $G$ that permutes the set of conference signals.
We call conference signals \textbf{equivalent} when they lie in the same $G$-orbit.
The \textbf{automorphism group} of a conference signal is its stabilizer in~$G$.
One can show that every element of $G$ takes the form $cC^j \mathcal{F}^k D_m$ for some unimodular $c \in \mathbb{C}$, $j,k \in \{0,1\}$, and $m \in \mathbb{Z}_n^\times$.
Thus, the automorphism group of a conference signal $f$ describes all the dilations $D_m$ or Fourier dilations $\mathcal{F} D_m$ for which $f$ is an eigenvector or a con-eigenvector.

\begin{table}[!b]
\begin{center}
{\small
\begin{tabular}{c|
  S[table-format=3.0]
  S[table-format=4.0]
  S[table-format=4.0]
  S[table-format=3.0]
  S[table-format=4.0]
  S[table-format=2.0]
  S[table-format=1.0]
  S[table-format=2.0]
  S[table-format=3.0]
  S[table-format=3.0]
  S[table-format=4.0]
}
$n$
& {inequiv}
& {num}
& {dil}
& {even}
& {odd}
& {neither}
& {$1$-eig}
& {$(-\mathrm{i})$-eig}
& {coneig}
& {odd$\wedge$coneig}
& {nonmult}
\\ \hline
\rule{0pt}{2.6ex}3  & 1   & 1    & 1    & 0   & 1    & 0  & 0 & 1  & 1   & 1   & 0\\
5  & 2   & 3    & 3    & 1   & 2    & 0  & 1 & 0  & 3   & 2   & 0\\
7  & 3   & 5    & 5    & 2   & 3    & 0  & 0 & 1  & 5   & 3   & 0\\
9  & 2   & 12   & 12   & 0   & 12   & 0  & 0 & 0  & 4   & 4   & {\dashentry}\\
11 & 6   & 19   & 19   & 4   & 15   & 0  & 0 & 1  & 11  & 7   & 10\\
13 & 13  & 79   & 79   & 11  & 60   & 8  & 3 & 0  & 29  & 20  & 68\\
15 & 12 & 96 & 96 & 0 & 88 & 8 & 0 & 8 & 24 & 20 & {\dashentry}\\
17 & 21  & 247  & 247  & 15  & 232  & 0  & 1 & 0  & 31  & 20  & 232\\
19 & 42  & 635  & 635  & 14  & 525  & 96 & 0 & 3  & 69  & 47  & 618\\
21 & 74  & 1324 & 1324 & 0   & 1324 & 0  & 0 & 0  & 84  & 84  & {\dashentry}\\
23 & 110 & 2925 & 2925 & 10  & 2915 & 0  & 0 & 13 & 141 & 131 & 2904\\
25 & 220 & 6720 & 6720 & 120 & 6600 & 0  & 4 & 0  & 200 & 188 & {\dashentry}
\end{tabular}
}
\end{center}

\caption[The results of a numerical search for conference signals over $\mathbb{Z}_n$.]{
The results of a numerical search for conference signals over $\mathbb{Z}_n$.
Here, ``inequiv'' gives the number of distinct equivalence classes we observed, ``num'' gives the number of normalized conference signals we observed, and the remaining columns give the number of normalized conference signals we observed with special properties: \\[3 pt]
\makebox[\linewidth][c]{%
\begin{tabular}{rlrl}
dil: & eigenvector of some nontrivial dilation \qquad &
$1$-eig: & $1$-eigenvector of the DFT \\
even: & even &
$(-\mathrm{i})$-eig: & $(-\mathrm{i})$-eigenvector of the DFT \\
odd: & odd &
coneig: & con-eigenvector of the DFT \\
neither: & neither even nor odd &
odd$\wedge$coneig: & odd con-eigenvector of the DFT \\[2 pt]
\multicolumn{4}{c}{nonmult: \quad not a multiple of a multiplicative character, where $n$ is prime}
\end{tabular}%
}}
\label{tbl: all small}
\end{table}

Finding numerical conference signals (including numerical counterexamples to Cohn's problem) was surprisingly easy with the following techniques:
\begin{itemize}
\item[(i)]
\textbf{Alternating projections.}
Initialize $f_0 \in \mathbb{C}^{\mathbb{Z}_n}$ randomly, then alternately ``fix'' the time and frequency domains by setting
\[
f_{k+1}(t) = 
\left\{ \begin{array}{ll}
\tfrac{f_k(t)}{|f_k(t)|} & \text{if }t \neq 0 \\ 
0 & \text{if }t = 0
\end{array}
\right\}
\qquad
\text{and then}
\qquad
\hat{f}_{k+2}(t) =
\left\{ \begin{array}{ll}
\tfrac{\sqrt{n}\hat{f}_{k+1}(t)}{|\hat{f}_{k+1}(t)|} & \text{if }t \neq 0 \\ 
0 & \text{if }t = 0
\end{array}
\right\}.
\]
This goes quickly, thanks to the fast Fourier transform.
There are other variations on this theme.
For example, after fixing the time domain with $f_{k+1}$, one can fix the frequency domain by making $f_{k+2}$ the projection of $f_{k+1}$ in an eigenspace or con-eigenspace of the DFT.
Other variations encourage (or force) the result to be odd or even. 
\item[(ii)]
\textbf{Energy minimization.}
Given $p > 1$, use any optimization method to solve:
\[
\begin{aligned}
& \text{minimize}
& & \sum_{s \neq 0} \big| \hat{f}(s) \big|^{2p} \\
& \text{subject to}
& & |f(t)|^2 = 1 \quad \forall t \neq 0, \\
&&& f(0)=\hat{f}(0)=0.
\end{aligned}
\]
Indeed, one can use Plancherel's theorem and H\"{o}lder's inequality to show the constraints imply $\sum_{s \neq 0} | \hat{f}(s)|^{2p} \geq n^p(n-1)$, with equality if and only if $f$ is a conference signal.
Sadly, the constraints $|f(t)|^2=1$ are non-convex, and local minima exist.
As a variation, one can optimize over a large domain by replacing the constraint ``$|f(t)|^2 = 1$ for all $t \neq 0$'' with ``$\sum_t |f(t)| = n-1$''; then the same lower bound holds, with the same conditions for equality.
\item[(iii)]
\textbf{Newton's method.}
Express the conditions $|f|^2 = 1-\delta_0$ and $f*f^* =n\delta_0 - 1$ from Theorem~\ref{thm: conference matrices}(ii) as polynomial constraints in the real and imaginary parts of $f$, and use the multivariate form of Newton's method to solve the system numerically.
In many cases, one can upgrade ``numerical conference signal'' to ``exact conference signal nearby'' by using Newton--Kantorovich techniques to prove that Newton's method converges to an exact solution in an explicit neighborhood.
See, for example, Corollary~3.4 in~\cite{CohnKM:16}.
\end{itemize}

\begin{example}
\label{ex: small conference signals}
The authors used techniques like those above to search for numerical (normalized) conference signals over $\mathbb{Z}_n$ for $n \leq 25$.
We do not claim to have found them all.
Table~\ref{tbl: all small} summarizes our findings, and its columns describe emergent features of the data.
For instance, for every odd $n \leq 25$, we found a numerical conference signal over $\mathbb{Z}_n$ that was an odd con-eigenvector of the DFT.
In Section~\ref{subsec: odd coneig}, we will prove such a conference signal exists for every odd~$n$.
We leave extensions of other emergent features as open problems, as below.
\end{example}

\begin{problem}
\label{prob: emergent features}
For each item below, find the answer, understand the proof, and present an intuitive explanation that others can understand.
\begin{itemize}
\item[(a)]
For each $n \equiv 3 \bmod 4$, does there exist a conference signal in the $(-\mathrm{i})$-eigenspace of the DFT?
\item[(b)]
Is every conference signal over $\mathbb{Z}_n$ an eigenvector of some nontrivial dilation?
\item[(c)]
When $n$ is even, are there no conference signals over $\mathbb{Z}_n$?
\item[(d)]
For each $n$, are there only finitely many normalized conference signals over $\mathbb{Z}_n$?
\end{itemize}
\end{problem}

In each case, all the evidence we know suggests the answer is ``yes''.
The conference signals in (a) are necessarily odd; the Legendre symbol is an example when $n$ is prime.
We do not know any other infinite families, but we have certified existence for all appropriate $n < 1000$ with Newton--Kantorovich techniques.
Finally, (d) is known to be true when $n$ is prime~\cite{Biro:99}.

\begin{example}
\label{ex: library}
The ancillary files for the arXiv version of this paper contain a library of numerical conference signals, organized into three parts.
The first part consists of the numerical conference signals from Example~\ref{ex: small conference signals} and Table~\ref{tbl: all small}.
The second part of the library gives a conference signal over $\mathbb{Z}_n$ that is an odd con-eigenvector of the DFT for each odd $n$ with $25 < n < 1000$.
The third part gives a conference signal over $\mathbb{Z}_n$ that lies in the $(-\mathrm{i})$-eigenspace of the DFT for each $n=3 \bmod 4$ with $25 < n < 1000$.
Every numerical conference signal of prime order in the second and third parts is chosen as a numerical counterexample to Cohn's problem.
Finally, the authors used Newton--Kantorovich techniques to prove that every odd member of the library (including every member of the second and third parts) has an exact normalized conference signal nearby, and furthermore, this conference signal is topologically isolated in the set of odd normalized conference signals.
\end{example}

\section{Existence}
\label{sec: existence}

In this section, we prove existence of conference signals nonconstructively.

\subsection{Benoist's technique}
\label{subsec: Benoist}

To begin, we illustrate Benoist's technique from~\cite{Benoist:24} by extending Proposition~\ref{prop: Benoist}(a).
(His proof of Proposition~\ref{prop: Benoist}(b) is much more complicated.)

\begin{theorem}
\label{thm: Benoist general}
Every finite abelian group of odd order admits an odd conference signal.
\end{theorem}

The nail above is easy to pound with the hammer below, which was forged independently by Cho~\cite{Cho:04} and the team of Biran, Entov, and Polterovich~\cite{BiranEP:04}.
In our phrasing, a vector $x \in \mathbb{C}^d$ is \textbf{flat} if $x \neq 0$ and $|x_i|$ is constant across $i \in [d]:=\{1,\ldots,d\}$.

\begin{proposition}[\cite{BiranEP:04,Cho:04}]
\label{prop: flat hammer}
Any unitary matrix sends some flat vector to a flat vector.
\end{proposition}

The proofs use symplectic geometry and Floer homology to study Hamiltonian diffeomorphisms of $\mathbb{C}P^{d-1}$ and their effects on the \textbf{Clifford torus}, i.e., the set of projective points with flat representatives.
Any unitary matrix induces a Hamiltonian diffeomorphism of $\mathbb{C}P^{d-1}$, and the authors of~\cite{BiranEP:04} and~\cite{Cho:04} show the Clifford torus cannot be displaced from itself by such a transformation.
In fact,~\cite{Cho:04} shows there are at least $2^{d-1}$ intersection points, provided the intersection is transverse.
This implies Proposition~\ref{prop: flat hammer} since an empty intersection would be transverse.

\begin{example}
\label{ex: Bloch flat}
Consider the case of a $2\times 2$ unitary.
As above, we work in projective space.
Here, $\mathbb{C}P^{1}$ can be visualized as the unit sphere in $\mathbb{R}^3$ by means of the \textbf{Bloch sphere}, where the north pole represents the span of $e_1$, the south pole represents the span of $e_2$, and $\left[ \begin{smallmatrix*} x & y & z \end{smallmatrix*} \right]^\top$ represents the image of the rank-one orthogonal projection $\tfrac{1}{2} \left[ \begin{smallmatrix*} 1 + z & x - y \mathrm{i} \\ x + y \mathrm{i} & 1 - z \end{smallmatrix*} \right]$.
Then the equator $z = 0$ represents the Clifford torus of flat projective points, as shown in Figure~\ref{fig: Bloch}.
The effect of any $2 \times 2$ unitary on projective space is to spin the ball (i.e., unit sphere; formally, there is a mapping $\operatorname{U}(2) \to \operatorname{SO}(3)$).
Spinning the ball moves the equator to some other great circle; since any two great circles intersect, some point on the equator moves to another point on the equator.
Hence, any $2 \times 2$ unitary moves some flat vector to another flat vector.
(In fact, there are two projectively distinct such vectors since the great circles intersect twice.)
\end{example}

\begin{remark}
As an immediate consequence of Proposition~\ref{prop: flat hammer}, every finite abelian group admits a CAZAC, i.e., a flat signal with flat Fourier transform.
\end{remark}

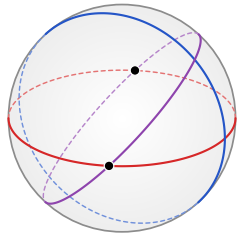
\begin{figure}[!b]
\begin{center}
\begin{tikzpicture}[
  x=1.5cm,y=1.5cm,
  line cap=round,line join=round,
  font=\small,
  visible arc/.style={line width=.8pt},
  hidden arc/.style={line width=.55pt,dash pattern=on 1.5pt off 1.3pt},
  intersection/.style={circle,fill=black,draw=white,
    line width=.45pt,minimum size=3.6pt,inner sep=0pt}
]
  \definecolor{equatorcolor}{HTML}{D62728}
  \definecolor{meridiancolor}{HTML}{2455C7}
  \definecolor{rotatedcolor}{HTML}{8E44AD}

  \def\viewazimuth{-65}
  \def\viewelevation{25}
  \pgfmathsetmacro{\sinazimuth}{sin(\viewazimuth)}
  \pgfmathsetmacro{\cosazimuth}{cos(\viewazimuth)}
  \pgfmathsetmacro{\sinelevation}{sin(\viewelevation)}
  \pgfmathsetmacro{\coselevation}{cos(\viewelevation)}

  \def\preparecircle#1#2#3#4#5#6{%
    \pgfmathsetmacro{\ux}{-(#1)*\sinazimuth+(#2)*\cosazimuth}%
    \pgfmathsetmacro{\uy}{-((#1)*\cosazimuth+(#2)*\sinazimuth)
                          *\sinelevation+(#3)*\coselevation}%
    \pgfmathsetmacro{\vx}{-(#4)*\sinazimuth+(#5)*\cosazimuth}%
    \pgfmathsetmacro{\vy}{-((#4)*\cosazimuth+(#5)*\sinazimuth)
                          *\sinelevation+(#6)*\coselevation}%
    \pgfmathsetmacro{\udepth}{((#1)*\cosazimuth+(#2)*\sinazimuth)
                              *\coselevation+(#3)*\sinelevation}%
    \pgfmathsetmacro{\vdepth}{((#4)*\cosazimuth+(#5)*\sinazimuth)
                              *\coselevation+(#6)*\sinelevation}%
    \pgfmathsetmacro{\phase}{atan2(\vdepth,\udepth)}%
  }
  \def\circlearc#1#2#3{%
    \draw[#1,domain=#2:#3,samples=101,variable=\t]
      plot ({\ux*cos(\t)+\vx*sin(\t)},
            {\uy*cos(\t)+\vy*sin(\t)});%
  }
  \def\equator{\preparecircle{1}{0}{0}{0}{1}{0}}
  \def\meridian{\preparecircle{1}{0}{0}{0}{0}{1}}
  \def\rotatedcircle{\preparecircle
    {1/sqrt(10)}{-3/sqrt(10)}{0}
    {9/sqrt(190)}{3/sqrt(190)}{10/sqrt(190)}}

  \shade[shading=radial,inner color=white,outer color=black!7]
    (0,0) circle[radius=1];
  \draw[black!45,line width=.65pt] (0,0) circle[radius=1];

  \equator
  \circlearc{equatorcolor!65,hidden arc}{\phase+90}{\phase+270}
  \meridian
  \circlearc{meridiancolor!65,hidden arc}{\phase+90}{\phase+270}
  \rotatedcircle
  \circlearc{rotatedcolor!65,hidden arc}{\phase+90}{\phase+270}

  \equator
  \circlearc{equatorcolor,visible arc}{\phase-90}{\phase+90}
  \meridian
  \circlearc{meridiancolor,visible arc}{\phase-90}{\phase+90}
  \rotatedcircle
  \circlearc{rotatedcolor,visible arc}{\phase-90}{\phase+90}

  \rotatedcircle
  \node[intersection] at (\ux,\uy) {};
  \node[intersection] at (-\ux,-\uy) {};
\end{tikzpicture}
\end{center}

\caption{
Any $2 \times 2$ unitary maps some flat vector to another flat vector, and it maps some real vector to a flat vector.
Indeed, in the Bloch sphere representation of $\mathbb{C}P^1$, the equator (red) represents flat vectors and the great circle $y=0$ (blue) represents real vectors.
Applying a unitary spins the ball, which sends either choice of red or blue to another great circle (purple).
Since any two great circles intersect, purple meets red somewhere.
That is, some red point moves to red, and some blue point moves to red.
}
\label{fig: Bloch}
\end{figure}

\begin{proof}[Proof of Theorem~\ref{thm: Benoist general}]
Let $\mathbb{C}_{\text{odd}}^A \leq \mathbb{C}^A$ and $\mathbb{C}_{\text{odd}}^{\hat{A}} \leq \mathbb{C}^{\hat{A}}$ be the spaces of odd functions on $A$ and $\hat{A}$, respectively.
Then, abbreviating $n:=|A|$, the normalized Fourier transform $f \mapsto \tfrac{1}{\sqrt{n}} \hat{f}$ determines a unitary transformation $\mathcal{F}_{\text{odd}} \colon \mathbb{C}_{\text{odd}}^A \to \mathbb{C}_{\text{odd}}^{\hat{A}}$.
We apply Proposition~\ref{prop: flat hammer} to a matrix representation of $\mathcal{F}_{\text{odd}}$ for a certain choice of bases.
Specifically, let $t_1,\ldots,t_{(n-1)/2}\in A$ consist of exactly one element from each pair $\{t,-t\}$ with $t \neq 0$, and put $f_j:= \tfrac{1}{\sqrt{2}}(\delta_{t_j}-\delta_{-{t_j}}) \in \mathbb{C}^A$ for each $j \in [(n-1)/2]$, where $\delta_s \in \mathbb{C}^A$ denotes the point mass at $s \in A$.
Then $f_1,\ldots,f_{(n-1)/2}$ is an orthonormal basis for $\mathbb{C}_{\text{odd}}^A$.
(Here we use the fact that $A$ has odd order, so that $\{t,-t\}$ has cardinality~2 whenever $t \neq 0$.)
Furthermore, $f\in \mathbb{C}_{\text{odd}}^A$ is flat on $A\setminus \{0\}$ if and only if its coordinate vector for this basis is flat.

Similarly, let $\alpha_1,\ldots,\alpha_{(n-1)/2} \in \hat{A}$ consist of exactly one element from each pair $\{\alpha,-\alpha\}$ with $\alpha \neq 0$, and define $\varphi_i := \tfrac{1}{\sqrt{2}}(\delta_{\alpha_i} - \delta_{-\alpha_i}) \in \mathbb{C}^{\hat{A}}$ for each $i \in [(n-1)/2]$.
Then $\varphi_1,\ldots,\varphi_{(n-1)/2}$ is an orthonormal basis for~$\mathbb{C}_{\text{odd}}^{\hat{A}}$, and $\varphi \in \mathbb{C}_{\text{odd}}^{\hat{A}}$ is flat on $\hat{A}\setminus\{0\}$ if and only if its coordinate vector for this basis is flat.
Applying Proposition~\ref{prop: flat hammer} to the matrix representation of $\mathcal{F}_{\text{odd}}$ in these bases, we obtain $f \in \mathbb{C}_{\text{odd}}^A$ that is flat on $A\setminus \{0\}$ and for which $\hat{f}$ is flat on $\hat{A}\setminus\{0\}$.
Both $f$ and $\hat{f}$ vanish at $0$ since they are odd; hence, a scalar multiple of~$f$ is an odd conference signal on~$A$.
\end{proof}

\subsection{Odd con-eigenvectors of the Fourier transform}
\label{subsec: odd coneig}

Next, we extend Proposition~\ref{prop: Benoist}(a) even further by showing every abelian group $A$ of odd order admits a conference signal that is an odd con-eigenvector of the Fourier transform.
To make sense of the term ``con-eigenvector'', we must first understand the Fourier transform as a mapping $\mathbb{C}^A \to \mathbb{C}^A$ rather than $\mathbb{C}^A \to \mathbb{C}^{\hat{A}}$.
This is possible because $A \cong \hat{A}$; however, our argument calls for an isomorphism with a little more structure.
Specifically, let $\mathbb{T} \leq \mathbb{C}^\times$ be the group of unimodular numbers.
Following Weil~\cite{Weil:64}, we call $\beta \colon A \times A \to \mathbb{T}$ a \textbf{bicharacter} if
\[
\beta(r,s+t) = \beta(r,s) \cdot \beta(r,t),
\qquad
\beta(r+s,t)=\beta(r,t) \cdot \beta(s,t)
\qquad
\forall r,s,t\in A.
\]
Equivalently, $\beta\colon A \times A \to \mathbb{T}$ is a bicharacter if it curries to a group homomorphism $A \to \hat{A} \colon s \mapsto (t \mapsto \beta(s,t))$.
We call $\beta$ \textbf{nondegenerate} when this curried map is an isomorphism $A \cong \hat{A}$; conversely, any such isomorphism gives a nondegenerate bicharacter.
As an additional property, a bicharacter $\beta$ is called \textbf{symmetric} if $\beta(s,t) = \beta(t,s)$ for all $s,t\in A$.

\begin{example}
\label{ex: Zn bichar}
For $A = \mathbb{Z}_n$, the mapping $\beta(s,t) = \omega_n^{st}$ is a symmetric nondegenerate bicharacter.
The same holds for the mapping $\beta_m(s,t) = \omega_n^{mst}$ whenever $m \in \mathbb{Z}_n^\times$.
\end{example}

\begin{remark}
Every finite abelian group $A$ admits a symmetric nondegenerate bicharacter.
To see this, first apply the fundamental theorem of finitely generated abelian groups to write $A = A_1 \times \cdots \times A_m$ with each $A_j \cong \mathbb{Z}_{n_j}$.
By Example~\ref{ex: Zn bichar}, there exist symmetric nondegenerate bicharacters $\beta_j \colon A_j \times A_j \to \mathbb{T}$ for all $j\in [m]$.
Define $\beta \colon A \times A \to \mathbb{T}$ by
$\beta(s,t) = \prod_{j=1}^m \beta_j(s_j,t_j)$, 
where $s_j,t_j \in A_j$ are the $A_j$-components of $s,t \in A$.
It is easy to see $\beta$ is a symmetric bicharacter, and one can prove it is nondegenerate by showing every member of $\hat{A} \cong \hat{A}_1 \times \cdots \times \hat{A}_m$ lies in the image of the curried map $A \to \hat{A} \colon s \mapsto (t \mapsto \beta(s,t))$.
\end{remark}

Nondegenerate bicharacters allow us to interpret the Fourier transform as a mapping from a space to itself, so that we can talk about its ``eigenvectors'' and ``con-eigenvectors''.
Specifically, given a nondegenerate bicharacter $\beta$ on a finite abelian group $A$ of order $n$, we write $\mathcal{F}_\beta \colon \mathbb{C}^A \to \mathbb{C}^A$ for the unitary transformation given by
\[
(\mathcal{F}_\beta f)(s)
= \tfrac{1}{\sqrt{n}} \hat{f}\bigl( t \mapsto \beta(s,t)\bigr)
= \tfrac{1}{\sqrt{n}} \sum_{t\in A} f(t) \overline{ \beta(s,t) }
\qquad
\forall f \in \mathbb{C}^A,\,
\forall s \in A.
\]
When viewed as a matrix, $\mathcal{F}_\beta \in \mathbb{C}^{A \times A}$ has entries $[ \mathcal{F}_\beta ]_{s,t} = \tfrac{1}{\sqrt{n}} \overline{\beta(s,t)}$, so that $\mathcal{F}_\beta$ is symmetric if and only if $\beta$ is symmetric.
Importantly, symmetry of $\beta$ implies $\mathcal{F}_\beta^2$ is the ``reversal'' operator $D_{-1}$; indeed, character orthogonality gives the last equality below:
\[
\bigl[ \mathcal{F}_\beta^2 \bigr]_{s,t}
= \tfrac{1}{n} \sum_{r\in A} \overline{\beta(s,r) \beta(r,t)}
= \tfrac{1}{n} \sum_{r\in A} \beta(-s,r) \overline{\beta(t,r)}
= \delta_{-s,t}
\qquad
\forall s,t \in A.
\]

\begin{example}
\label{ex: dilated Fourier}
Let $A = \mathbb{Z}_n$, and let $\beta$ and $\beta_m$ be as in Example~\ref{ex: Zn bichar}, where $m \in \mathbb{Z}_n^\times$.
Then $\mathcal{F}_\beta = \mathcal{F}$ is the usual DFT, while $\mathcal{F}_{\beta_m} = D_{m^{-1}} \mathcal{F}=\mathcal{F} D_m$.
\end{example}

\begin{theorem}
\label{thm: odd coneig existence}
Let $A$ be a finite abelian group of odd order, and let $\beta$ be a symmetric nondegenerate bicharacter on~$A$.
Then there exists an odd conference signal $f$ on $A$ for which $\mathcal{F}_\beta f = \overline{f}$.
\end{theorem}

Our proof technique is similar to Benoist's, relying on a nontrivial result of symplectic geometry, this one thanks to Alston and Amorim.

\begin{proposition}[\cite{AlstonA:12}]
\label{prop: real flat hammer}
Any unitary matrix sends some real vector to a flat vector.
\end{proposition}

Indeed,~\cite{AlstonA:12} shows that, under any Hamiltonian diffeomorphism of $\mathbb{C}P^{d-1}$, the image of $\mathbb{R}P^{d-1}$ intersects the Clifford torus in at least $2^{\lceil{(d-1)/2}\rceil}$ points, provided the intersection is transverse.
This implies Proposition~\ref{prop: real flat hammer} since an empty intersection would be transverse.

\begin{example}
Consider a $2\times 2$ unitary again, as in Example~\ref{ex: Bloch flat}.
In the Bloch sphere representation of $\mathbb{C}P^1$, the set $\mathbb{R}P^1$ of projective points with at least one real representative corresponds to the great circle $y=0$, as in Figure~\ref{fig: Bloch}.
As before, applying any $2 \times 2$ unitary spins the ball, which moves the great circle $y=0$ to some other great circle.
That other great circle intersects the equator somewhere, so the unitary moves some real vector to a flat vector.
\end{example}

\begin{proof}[Proof of Theorem~\ref{thm: odd coneig existence}]
As in the proof of Theorem~\ref{thm: Benoist general}, $\mathcal{F}_\beta$ restricts to a unitary transformation $\mathcal{F}_{\beta,\text{odd}} \colon \mathbb{C}_{\text{odd}}^A \to \mathbb{C}_{\text{odd}}^A$; furthermore, $\mathcal{F}_{\beta,\text{odd}}^2 = -I$ since $\mathcal{F}_\beta^2 = D_{-1}$.
Consequently,
$U:= 
\tfrac{1}{\sqrt{2}}(I -  \mathcal{F}_{\beta,\text{odd}})
$
is a unitary transformation of $\mathbb{C}_{\text{odd}}^A$.
We will apply Proposition~\ref{prop: real flat hammer} to its matrix representation in a certain basis.

Specifically, put $n:=|A|$, and equip $\mathbb{C}_{\text{odd}}^A$ with the orthonormal basis $f_1,\ldots,f_{(n-1)/2}$ from the proof of Theorem~\ref{thm: Benoist general}.
As in that proof, any $f \in \mathbb{C}_{\text{odd}}^A$ is flat on $A \setminus \{0\}$ if and only if its coordinate vector is flat; furthermore, any $g \in \mathbb{C}_{\text{odd}}^A$ is real-valued if and only if its coordinate vector is real.
Applying Proposition~\ref{prop: real flat hammer} to the matrix representation of $U$, we obtain a real-valued $g \in \mathbb{C}_{\text{odd}}^A$ for which
\[
f:= Ug 
= \tfrac{1}{\sqrt{2}}(g - \mathcal{F}_\beta g) 
\in \mathbb{C}_{\text{odd}}^A
\]
is flat on $A \setminus \{0\}$.
Here, $f(0) = 0$ since $f$ is odd.
To prove the theorem, it therefore suffices to show $\mathcal{F}_\beta f = \overline{f}$.
Indeed, $\mathcal{F}_\beta g$ takes purely imaginary values since $g$ is odd and real-valued, so $\operatorname{Re} f = \tfrac{1}{\sqrt{2}}g$, $\operatorname{Im} f = \tfrac{\mathrm{i}}{\sqrt{2}} \mathcal{F}_\beta g$, and
\[
\mathcal{F}_\beta f
= \tfrac{1}{\sqrt{2}}(\mathcal{F}_\beta g - \mathcal{F}_\beta^2 g)
= \tfrac{1}{\sqrt{2}}(g + \mathcal{F}_\beta g)
= \overline{f}.
\qedhere
\]
\end{proof}

\begin{corollary}
For any odd $n>1$ and any $m \in \mathbb{Z}_n^\times$, there exists an odd conference signal $f$ on $\mathbb{Z}_n$ that satisfies $\hat{f} = \sqrt{n}D_m \overline{f}$.
\end{corollary}

Indeed, this follows immediately from Example~\ref{ex: dilated Fourier} and Theorem~\ref{thm: odd coneig existence}.

\section{Explicit construction (composite order)}
\label{sec: explicit composite}

The results from the last section are a little unsatisfying for the authors because they leave us dying to know:
\textit{what are those conference signals}?
Our urge to answer this question was made even stronger by the multitude of numerical conference signals we observed with mysterious entries that did not seem to align with any of the usual suspects for unimodular numbers (such as roots of unity or Gauss sums).
In this section, which contains the first of our main results, 
we crack into this mystery
by providing the first explicit construction of conference signals with composite order.
Specifically, we construct conference signals on $\mathbb{Z}_n$ whenever $n = 3p$ for some prime $p > 3$.
In the theorem statement below, $\mathbf{1}_3,\mathbf{1}_p \colon \mathbb{Z}_n \to \{0,1\}$ are the indicator functions of the subgroups with orders $3$ and $p$, respectively:
\[
\mathbf{1}_3(t) = \left\{ \begin{array}{ll}
1 & \text{if $t = 0 \bmod p$} \\
0 & \text{otherwise}
\end{array} \right\},
\qquad
\mathbf{1}_p(t) = \left\{ \begin{array}{ll}
1 & \text{if $t=0 \bmod 3$} \\
0 & \text{otherwise}
\end{array} \right\}
\qquad
\forall t \in \mathbb{Z}_n.
\]
We also extend the Legendre symbols on $\mathbb{Z}_3$ and $\mathbb{Z}_p$ to define functions
$\psi_3,\psi_p \colon \mathbb{Z}_n \to \mathbb{R}$ given by
\[
\psi_3(t) = \chi_3(t),
\qquad
\psi_p(t) = \chi_p(t)
\qquad
\forall
t \in \mathbb{Z}_n,
\]
where the values above depend only on $t \bmod n$; in particular, $\psi_3$ and $\psi_p$ are multiplicative.

\begin{theorem}
\label{thm: n = 3p}
Assume $n = 3p$ where $p>3$ is prime.
In the notation above, there exist $a,b,c,d\in \mathbb{C}$ for which
\begin{equation}
\label{eq: f for 3p}
f := (a+b\mathbf{1}_3)\psi_3 + (c+d\mathbf{1}_p)\psi_p
\end{equation}
is a conference signal on $\mathbb{Z}_n$ with $\hat{f} = \sqrt{n} \overline{f}$, where multiplication in~\eqref{eq: f for 3p} is entrywise.

Specifically, put $\varepsilon := \chi_3(p) \in \{ \pm 1 \}$, and consider the polynomials
\[
g(x) = (p+1)x^2-4 \varepsilon \sqrt{p} x + (p+1)
\]
and
\[h(x) = \begin{cases}
\bigl( 4-2\varepsilon\sqrt{3} \bigr)x^2 + \bigl( 4+2\varepsilon\sqrt{3} \bigr)
& \text{if }p = 1 \bmod 4, \\[3 pt]
4x^2 - 4\varepsilon \sqrt{3} x + 4
&\text{if }p = 3 \bmod 4.
\end{cases}
\]
Then there exists $x \in \mathbb{R}$ satisfying
\begin{equation}
\label{eq: degree four}
g(x)h(x) = (x^2+1)\bigl( g(x) + h(x) \bigr),
\end{equation}
and for any such~$x$ and either choice of $\delta \in \{\pm1\}$, one can take $f$ to be given by~\eqref{eq: f for 3p} with
\begin{align}
\label{eq: abcd}
a &= \delta\tfrac{-x+\mathrm{i}}{\sqrt{g(x)}},
&
b &=\varepsilon\mathrm{i} \sqrt{p}\cdot \overline{a}, \\
c &= \tfrac{1+x \mathrm{i}}{\sqrt{h(x)}},
&
d &= \begin{cases}
\varepsilon\sqrt{3} \cdot \overline{c} & \text{if }p = 1 \bmod 4, \\[3 pt]
- \varepsilon\mathrm{i}\sqrt{3} \cdot \overline{c} & \text{if }p = 3 \bmod 4,\notag
\end{cases}
\end{align}
where it holds automatically that $g(x),h(x)>0$.
\end{theorem}

One can use Ferrari's method to solve the degree-four polynomial equation~\eqref{eq: degree four} explicitly in terms of radicals, but the result is not particularly illuminating.

\begin{remark}
If $p = 3 \bmod 4$, then $\chi_p$ is odd and so $f$ in~\eqref{eq: f for 3p} is odd.
However, if $p = 1 \bmod 4$, then $\chi_p$ is even and one can show that the conference signal $f$ above is neither even nor odd, and furthermore, $f^* \notin\{f,-f\}$.
Consequently, the complex conference matrices in Theorem~\ref{thm: conference matrices}(iii) are neither symmetric, skew-symmetric, Hermitian, nor skew-Hermitian.
Another complex conference matrix with this property can be obtained by taking $b=c=\mathrm{i}$ in equation~(22) of~\cite{Dita:09}.
By contrast, every real normalized conference matrix is either symmetric or skew-symmetric~\cite{DelsarteGS:71}.
\end{remark}

\begin{remark}
By Corollary~\ref{cor: EITFF}, the construction above leads to an explicit $\operatorname{EITFF}(3p-1,2,3p)$ for every prime $p > 3$ with $p = 3 \bmod 4$.
To the authors' knowledge, these are all new EITFF parameters.
\end{remark}

\begin{proof}[Proof of Theorem~\ref{thm: n = 3p}]
The argument is long, so we organize it into three parts.

\textbf{Part 1.}
First, given arbitrary $a,b,c,d\in \mathbb{C}$, we compute the Fourier transform of each term in
\[
f = a \psi_3 + b \mathbf{1}_3 \psi_3 + c \psi_p + d \mathbf{1}_p \psi_p
\]
to find conditions under which $\hat{f} = \sqrt{n} \overline{f}$.
For any $t \in \mathbb{Z}_n$,
\[
\widehat{\psi}_3(t)
= \sum_{r=0}^2 \sum_{s=0}^{p-1} \psi_3(r+3s) \omega_n^{-(r+3s)t}
= \sum_{r=0}^2 \chi_3(r) \omega_n^{-rt} \sum_{s=0}^{p-1}   \omega_p^{-st}.
\]
When $t \neq 0 \bmod p$, this vanishes; when instead $t = pk$ for some $k\in\{0,1,2\}$, it equals
\[
\widehat{\psi}_3(pk) = p \sum_{r=0}^2 \chi_3(r) \omega_n^{-rpk}
= p \sum_{r=0}^2 \chi_3(r) \omega_3^{-rk}
= p \widehat{\chi}_3(k)
= - \mathrm{i} p \sqrt{3} \chi_3(k),
\]
where the last equality applies~\eqref{eq: Legendre Fourier}.
Since $\psi_3(pk) = \psi_3(p) \psi_3(k) = \varepsilon \chi_3(k)$, the above says
\[
\widehat{\psi}_3(pk) = - \varepsilon \mathrm{i} p \sqrt{3} \psi_3(pk).
\]
Combining the cases, we have
\begin{equation}
\label{eq: psi3 hat}
\widehat{\psi}_3 = - \varepsilon \mathrm{i} p \sqrt{3} \mathbf{1}_3 \psi_3.
\end{equation}
Taking the Fourier transform of both sides gives
\[
n \psi_3(-t) = - \varepsilon \mathrm{i} p \sqrt{3} \widehat{\mathbf{1}_3 \psi_3}(t)
\qquad
\forall t \in \mathbb{Z}_n,
\]
and since $\psi_3$ is odd,
\begin{equation}
\label{eq: 13 psi3 hat}
\widehat{\mathbf{1}_3 \psi_3} = - \varepsilon \mathrm{i} \sqrt{3} \psi_3.
\end{equation}

For the next term, put
\[
\lambda:= \begin{cases}
1 & \text{if } p = 1 \bmod 4, \\
-\mathrm{i} & \text{if } p = 3 \bmod 4,
\end{cases}
\]
so that~\eqref{eq: Legendre Fourier} reads $\widehat{\chi}_p = \lambda \sqrt{p} \chi_p$.
Then calculations like those above show $\widehat{\psi}_p(t) = 0$ whenever $t \neq 0 \bmod 3$ and $\widehat{\psi}_p(3k) = 3 \lambda \sqrt{p} \chi_p(k)$ for each $k \in \{0,\ldots,p-1\}$.
To express the right-hand side in terms of $\psi_p(3k)$, observe that $\psi_p(3k) = \chi_p(3) \chi_p(k)$, while quadratic reciprocity gives
\[
\chi_p(3) \chi_3(p)
= \left\{ \begin{array}{rl}
1 & \text{if } p = 1 \bmod 4 \\
-1 & \text{if } p = 3 \bmod 4
\end{array} \right\}
= \lambda^2
\]
so that $\chi_p(3) = \varepsilon \lambda^2$.
Consequently, $\widehat{\psi}_p(3k) = 3 \varepsilon \overline{\lambda} \sqrt{p} \psi_p(3k)$ and
\begin{equation}
\label{eq: psip hat}
\widehat{\psi}_p = 3 \varepsilon \overline{\lambda} \sqrt{p} \mathbf{1}_p \psi_p.
\end{equation}
Taking the Fourier transform of both sides gives
\[
n \psi_p(-t) = 3 \varepsilon \overline{\lambda} \sqrt{p} \widehat{ \mathbf{1}_p \psi_p }(t)
\qquad
\forall t \in \mathbb{Z}_n,
\]
and since $\psi_p(-t) = \chi_p(-1) \psi_p(t) = \overline{\lambda}^2 \psi_p(t)$,
\begin{equation}
\label{eq: 1p psip hat}
\widehat{ \mathbf{1}_p \psi_p }
= \varepsilon \overline{\lambda} \sqrt{p} \psi_p.
\end{equation}
Now
\eqref{eq: psi3 hat}, 
\eqref{eq: 13 psi3 hat},
\eqref{eq: psip hat}, and
\eqref{eq: 1p psip hat}
combine to show
\begin{align*}
\hat{f} &= 
a\bigl(-\varepsilon \mathrm{i} p \sqrt{3}\bigr) \mathbf{1}_3 \psi_3
+ b\bigl(-\varepsilon \mathrm{i} \sqrt{3}\bigr) \psi_3
+ c\bigl(3 \varepsilon \overline{\lambda} \sqrt{p} \bigr) \mathbf{1}_p \psi_p
+ d \bigl( \varepsilon \overline{\lambda} \sqrt{p} \bigr) \psi_p \\[2 pt]
&= \sqrt{n} \bigl( \tfrac{-b\varepsilon \mathrm{i}}{\sqrt{p}} - a\varepsilon \mathrm{i}  \sqrt{p} \mathbf{1}_3 \bigr) \psi_3
+ \sqrt{n} \bigl( \tfrac{d \varepsilon \overline{\lambda}}{\sqrt{3}} + c \varepsilon \overline{\lambda} \sqrt{3} \mathbf{1}_p \bigr) \psi_p.
\end{align*}
Comparing with~\eqref{eq: f for 3p}, we find $\hat{f} = \sqrt{n} \overline{f}$ whenever
\begin{equation}
\label{eq: b and d}
b = \varepsilon \mathrm{i} \sqrt{p} \cdot \overline{a},
\qquad
d = \varepsilon \lambda \sqrt{3} \cdot \overline{c}.
\end{equation}

\textbf{Part 2.}
Assume $b$ and $d$ are as in~\eqref{eq: b and d}.
We show there exist $a,c\in\mathbb{C}$ for which $f$ in~\eqref{eq: f for 3p} is unimodular on $\mathbb{Z}_n \setminus \{0\}$; any such $f$ is a conference signal since $\hat{f} = \sqrt{n} \overline{f}$.
First observe that $\mathbf{1}_3(t) = 1$ if and only if $\psi_p(t) = 0$, and likewise $\mathbf{1}_p(t) = 1$ if and only if $\psi_3(t) = 0$; hence the value of $f(t)$ is completely determined by $(\psi_3(t),\psi_p(t))$.
By iterating through all nine possible values of $(\psi_3(t),\psi_p(t))$, and by observing that $\psi_3(t)=\psi_p(t) = 0$ if and only if $t =0$, we find $f(t)$ takes the following values on $\mathbb{Z}_n \setminus \{0\}$:
\[
\pm(a+b),
\,
\pm(c+d),
\,
\pm(a+c),
\,
\pm(a-c).
\]
Substituting $b$ and $d$ from~\eqref{eq: b and d}, we see that $f$ is unimodular on $\mathbb{Z}_n \setminus \{0\}$ if and only if
\[
\bigl| a+\varepsilon \mathrm{i} \sqrt{p} \cdot\overline{a} \bigr|^2
= \bigl|c+\varepsilon \lambda \sqrt{3} \cdot \overline{c} \bigr|^2
= |a+c|^2
= |a-c|^2
= 1.
\]

\begin{figure}[!bt]
\centering

\definecolor{ccol}{HTML}{2455C7}
\definecolor{acol}{HTML}{D62728}
\definecolor{plotcol}{RGB}{0,0,0}

\tikzset{
  vec/.style={-{Stealth[length=3.5pt]},line width=0.9pt},
  ellipsea/.style={acol,line width=1pt},
  ellipsec/.style={ccol,line width=1pt},
  complexaxis/.style={
    -{Stealth[length=3pt]},black!48,line width=0.4pt
  },
  every node/.style={font=\small}
}

\makebox[\linewidth][c]{%
\begin{tikzpicture}[
  x=1.50cm,
  y=1.50cm,
  baseline=(current bounding box.center)
]
  \draw[complexaxis] (-1.53,0)--(1.65,0)
    node[
      anchor=west,
      font=\footnotesize,
      text=black,
      inner sep=2pt
    ] {$\operatorname{Re}$};
  \draw[complexaxis] (0,-1.08)--(0,1.22)
    node[
      anchor=north west,
      font=\footnotesize,
      text=black,
      inner sep=2pt
    ] {$\operatorname{Im}$};
  \def\figuretheta{18}
  \pgfmathsetmacro{\radiusc}{
    1/sqrt(4-2*sqrt(3)*cos(2*\figuretheta))
  }
  \pgfmathsetmacro{\radiusa}{
    1/sqrt(6+2*sqrt(5)*sin(2*\figuretheta))
  }
  \coordinate (c) at
    ({\radiusc*cos(\figuretheta)},{\radiusc*sin(\figuretheta)});
  \coordinate (a) at
    ({-\radiusa*sin(\figuretheta)},{\radiusa*cos(\figuretheta)});
  \draw[ellipsec] (0,0)
    ellipse[x radius=1.366025404,y radius=0.366025404];
  \draw[ellipsea,rotate=-45] (0,0)
    ellipse[x radius=0.309016994,y radius=0.809016994];
  \draw[vec,ccol] (0,0)--(c);
  \draw[vec,acol] (0,0)--(a);
  \draw[black!60,line width=0.5pt]
    ({0.085*cos(\figuretheta)},{0.085*sin(\figuretheta)})--
    ({0.085*(cos(\figuretheta)-sin(\figuretheta))},
     {0.085*(sin(\figuretheta)+cos(\figuretheta))})--
    ({-0.085*sin(\figuretheta)},{0.085*cos(\figuretheta)});
  \node[ccol,anchor=west,inner sep=0.5pt]
    at (0.97,0.34) {$c$};
  \node[acol,inner sep=0pt]
    at (-0.300,0.270) {$a$};
  \draw[black!60,line width=0.5pt] (0.70,0)
    arc[start angle=0,end angle=\figuretheta,radius=0.70];
  \node[font=\footnotesize,inner sep=0.5pt]
    at (0.900,0.14) {$\theta$};
  \node[ccol,anchor=west,font=\footnotesize,inner sep=1pt]
    at (-1.18,-0.38) {$E_c$};
  \node[acol,anchor=west,font=\footnotesize,inner sep=1pt]
    at (0.58,0.61) {$E_a$};
\end{tikzpicture}%
\hspace{1cm}%
\begin{tikzpicture}[baseline=(current bounding box.center)]
\begin{axis}[
  width=2.75in,
  height=1.67in,
  axis lines=left,
  axis line style={black!55},
  tick style={black!55},
  xmin=0,
  xmax=2*pi,
  ymin=0,
  ymax=2.20,
  xtick={0,pi/2,pi,3*pi/2},
  xticklabels={,$\frac\pi2$,$\pi$,$\frac{3\pi}2$},
  ytick={0,1,2},
  yticklabels={,$1$,$2$},
  tick label style={font=\footnotesize,text=black},
  xlabel={$\theta$},
  label style={font=\small,text=black},
  xlabel style={
    at={(axis description cs:1.025,0)},
    anchor=west,
    text=black,
    inner xsep=0pt,
    xshift=.3333em
  },
  clip=false
]
  \addplot[black!60,dashed,domain=0:2*pi] {1};
  \addplot[
    plotcol,
    line width=1.05pt,
    domain=0:2*pi,
    samples=1441
  ] {
    1/(6+2*sqrt(5)*sin(deg(2*x)))
    +1/(4-2*sqrt(3)*cos(deg(2*x)))
  };
  \addplot[
    plotcol,
    only marks,
    mark=*,
    mark size=2.8pt,
    mark options={fill=black,draw=white,line width=.55pt}
  ] coordinates {
    (0.2979991607327084,1)
    (2.7509951554122467,1)
    (3.4395918143225015,1)
    (5.8925878090020403,1)
  };
  \node[font=\small,anchor=south,text=black]
    at (rel axis cs:0,0.975) {$F(\theta)$};
  \node[
    font=\footnotesize,
    anchor=north,
    text=black,
    yshift=-2pt
  ] at (axis cs:2*pi,0) {$2\pi$};
\end{axis}
\end{tikzpicture}%
}\par

\caption{
{\scriptsize\textbf{Left:}}
When $b$ and $d$ satisfy~\eqref{eq: b and d}, $f$ is a conference signal if and only if $a$ and $c$ are orthogonal, lie on the ellipses $E_a$ and $E_c$, and satisfy $|a|^2+|c|^2=1$.
{\scriptsize\textbf{Right:}}
For $\theta \in \mathbb{R}$, let $a$ and $c$ be as depicted on the left, and put $F(\theta)=|a|^2+|c|^2$.
Then $F(\theta)>|c|^2>1$ when $c$ points along the major axis, and $F(\theta)<1$ somewhere since its average value on $[0,2\pi)$ is the sum of the ellipse areas divided by $\pi$, which is less than~$1$.
By the intermediate value theorem, there exists $\theta$ with $F(\theta)=1$.
}
\label{fig: geometry}
\end{figure}
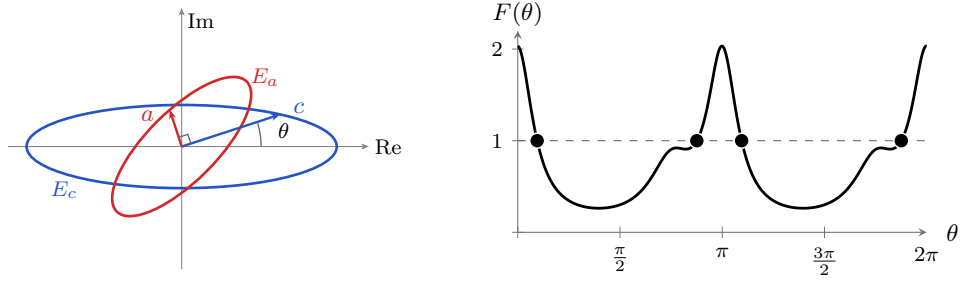

These constraints have a beautiful geometric interpretation.
Namely, the first two constraints say $a$ and $c$ lie on the ellipses
\[
E_a := \bigl\{ z \in \mathbb{C} : |z + \varepsilon \mathrm{i} \sqrt{p} \cdot \overline{z}|^2 = 1 \bigr\},
\qquad
E_c := \bigl\{ z \in \mathbb{C} : |z + \varepsilon \lambda \sqrt{3} \cdot \overline{z}|^2 = 1 \bigr\},
\]
while the last two constraints say $a$ and $c$ are orthogonal with $|a|^2+|c|^2=1$, as seen by adding and subtracting the equations
\begin{align*}
|a+c|^2 &= |a|^2 + 2 \operatorname{Re} (a \overline{c}) + |c|^2, \\[3 pt]
|a-c|^2 &= |a|^2 - 2 \operatorname{Re} (a \overline{c}) + |c|^2.
\end{align*}
See Figure~\ref{fig: geometry}.
To show such $a,c$ exist, let
\[
z_a(\theta) = \bigl| z_a(\theta) \bigr| e^{\theta \mathrm{i}},
\qquad
z_c(\theta) = \bigl| z_c(\theta) \bigr| e^{\theta \mathrm{i}}
\]
give polar parameterizations of $E_a$ and $E_c$ for $\theta \in [0,2\pi)$; here, $|z_a|$ and $|z_c|$ are $\pi$-periodic functions on $\mathbb{R}$.
For some $\theta$ and $\delta \in \{\pm 1\}$, we will take $c = z_c(\theta)$ and $a = \delta z_a(\theta+\tfrac{\pi}{2})$, so that $|a|^2+|c|^2$ is given by
\[
F(\theta):= \bigl| z_a(\theta+\tfrac{\pi}{2}) \bigr|^2 + \bigl| z_c(\theta) \bigr|^2,
\qquad
\theta \in \mathbb{R}.
\]
Then $f$ is a conference signal if and only if $F(\theta)=1$.

We will use the intermediate value theorem to show such $\theta$ exists.
First, we determine the geometry of $E_a$ and $E_c$.
To treat them simultaneously, pick $r>1$ and $\varphi \in \mathbb{R}$, and consider
\[
E:=\bigl\{ z \in \mathbb{C} : | z + e^{\varphi \mathrm{i}} r \overline{z} |^2 = 1 \bigr\}.
\]
The rotation $w = e^{-\varphi \mathrm{i}/2} z$ satisfies
\[
w+r\overline{w} = e^{-\varphi \mathrm{i}/2}\bigl(z+e^{\varphi \mathrm{i}} r \overline{z}\bigr),
\]
so $E$ is a rotation of $E':=\bigl\{ w \in \mathbb{C} : |w+r\overline{w}|^2 = 1 \bigr\}$.
For $w=u+v \mathrm{i}$ with $u,v \in \mathbb{R}$, we have
\begin{equation}
\label{eq: w abs}
|w + r \overline{w} |^2= (1+r)^2u^2 + (1-r)^2v^2,
\end{equation}
so $E'$ and $E$ have semiaxis lengths $1/(r-1)$ and $1/(r+1)$, for area $\pi/(r^2-1)$.
In particular, $E_c$ has major semiaxis length $1/(\sqrt{3}-1)>1$, and when $z_c(\theta)$ points along that major semiaxis we have $F(\theta) > |z_c(\theta)|^2 > 1$.
To see that $F(\theta)< 1$ somewhere, apply the polar area formula to find the average value of $F$ on $[0,2\pi)$ is the sum of the areas of $E_a$ and~$E_c$, divided by $\pi$:
\[
\tfrac{1}{2\pi} \int_0^{2\pi} F(\theta) \, d \theta
= \tfrac{1}{\pi} \int_0^{2\pi} \tfrac{1}{2} \bigl| z_a(\theta+\tfrac{\pi}{2}) \bigr|^2\, d\theta 
+ \tfrac{1}{\pi} \int_0^{2\pi} \tfrac{1}{2} \bigl| z_c(\theta) \bigr|^2\, d\theta 
= \frac{1}{2} + \frac{1}{p-1}
< 1.
\]
Consequently, $F(\theta)<1$ for some $\theta$; by the intermediate value theorem, there exists $\theta$ with $F(\theta)=1$.

\smallskip

\textbf{Part 3.}
Finally, we show the change of variables $x = \tan \theta$ turns the equation $F(\theta)=1$ into~\eqref{eq: degree four}.
To begin, we compute $z_a(\theta)$ and $z_c(\theta)$ explicitly.
Considering~\eqref{eq: w abs}, $E'$ has polar parameterization
\[
w(\theta)
= \frac{ \cos \theta + \mathrm{i} \sin \theta}{\sqrt{ (1+r)^2 \cos^2 \theta + (1-r)^2 \sin^2 \theta}}
= \frac{e^{\theta \mathrm{i}}}{\sqrt{1+r^2+2r\cos 2 \theta}}
\]
for $\theta\in [0,2\pi)$.
By the change of variables $z=e^{\varphi \mathrm{i}/2} w$, $E$ has polar parameterization
\[
z(\theta) = e^{\varphi \mathrm{i}/2} w(\theta-\tfrac{\varphi}{2})
= \frac{ e^{\theta \mathrm{i}} }{\sqrt{1+r^2+2r\cos(2\theta - \varphi)}}
\]
for $\theta \in [0,2\pi)$.
To interpret this for $E_a$, take $r=\sqrt{p}$ and $\varphi = \tfrac{\varepsilon \pi}{2}$.
Since $\cos(2\theta-\tfrac{\varepsilon \pi}{2}) = \varepsilon \sin(2\theta)$,
\[
z_a(\theta) = \frac{e^{\theta \mathrm{i}}}{\sqrt{1+p+2\varepsilon\sqrt{p}\sin(2\theta)}}
\]
and
\[
z_a(\theta+\tfrac{\pi}{2})
= \frac{\mathrm{i}e^{\theta \mathrm{i}}}{\sqrt{1+p-2\varepsilon\sqrt{p}\sin(2\theta)}}.
\]
For $E_c$, we take $r = \sqrt{3}$ and choose $\varphi$ depending on $p$ to ensure $e^{\varphi \mathrm{i}}=\varepsilon \lambda$.
Namely, when $p=1 \bmod 4$ we take $\varphi = \tfrac{(1-\varepsilon)\pi}{2}$, so that $\cos(2\theta - \varphi) = \varepsilon \cos(2\theta)$; when $p=3 \bmod 4$ we take $\varphi = -\tfrac{\varepsilon \pi}{2}$, so that $\cos(2\theta - \varphi) = - \varepsilon \sin(2\theta)$.
Overall,
\[
z_c(\theta)
= \begin{cases}
\frac{e^{\theta \mathrm{i}}}{\sqrt{4+2\varepsilon\sqrt{3}\cos(2\theta)}} & \text{if $p = 1 \bmod 4$}, \\[8 pt]
\frac{e^{\theta \mathrm{i}}}{\sqrt{4-2\varepsilon\sqrt{3}\sin(2\theta)}} & \text{if $p = 3 \bmod 4$}.
\end{cases}
\]
Therefore,
\[
F(\theta)
= \bigl| z_a(\theta+\tfrac{\pi}{2}) \bigr|^2 + \bigl| z_c(\theta) \bigr|^2
= \frac{1}{1+p-2\varepsilon\sqrt{p}\sin(2\theta)}
+ \begin{cases}
\frac{1}{4+2\varepsilon\sqrt{3}\cos(2\theta)} & \text{if $p = 1 \bmod 4$}, \\[8 pt]
\frac{1}{4-2\varepsilon\sqrt{3}\sin(2\theta)} & \text{if $p = 3 \bmod 4$}.
\end{cases}
\]
Since $F$ is $\pi$-periodic, some solution of the equation $F(\theta)=1$ lies in the interval $[-\tfrac{\pi}{2},\tfrac{\pi}{2}]$; furthermore, it does not occur at either endpoint since
\[
F(\pm \tfrac{\pi}{2})
= \tfrac{1}{1+p}+\begin{cases}
\tfrac{1}{4-2\varepsilon\sqrt{3}} & \text{if $p = 1 \bmod 4$}, \\[8 pt]
\tfrac{1}{4} & \text{if $p = 3 \bmod 4$},
\end{cases}
\]
so that $F(\pm \tfrac{\pi}{2})$ is irrational when $p=1\bmod 4$ and $F(\pm \tfrac{\pi}{2})<1$ when $p=3\bmod 4$.
Thus, the equation $F(\theta)=1$ has a solution in $(-\tfrac{\pi}{2},\tfrac{\pi}{2})$.
Restrict $\theta$ to that interval and put $x = \tan \theta$, so that $\theta = \arctan x$ and
\[
\sin \theta = \frac{x}{\sqrt{1+x^2}},
\quad\,
\cos \theta = \frac{1}{\sqrt{1+x^2}},
\quad\,
\sin(2\theta)=\frac{2x}{1+x^2},
\quad\,
\cos(2\theta)=\frac{1-x^2}{1+x^2},
\quad\,
e^{\theta \mathrm{i}} = \frac{1+x\mathrm{i}}{\sqrt{1+x^2}}.
\]
Then
\[
z_a(\theta+\tfrac{\pi}{2})
= \frac{-x+\mathrm{i}}{\sqrt{1+x^2}}\cdot\frac{1}{\sqrt{1+p-2\varepsilon \sqrt{p} \cdot \tfrac{2x}{1+x^2}}}
= \frac{-x+\mathrm{i}}{\sqrt{(1+x^2)(1+p)-4\varepsilon \sqrt{p} x}}
= \frac{-x+\mathrm{i}}{\sqrt{g(x)}}
\]
and similarly
\[
z_c(\theta) = \left\{ \begin{array}{ll}
\tfrac{1+x\mathrm{i}}{\sqrt{4(1+x^2)+2\varepsilon\sqrt{3}(1-x^2)}} & \text{if $p=1\bmod 4$} \\[8 pt]
\tfrac{1+x\mathrm{i}}{\sqrt{4(1+x^2)-4 \varepsilon \sqrt{3}x}} & \text{if $p=3\bmod 4$}
\end{array} \right\}
= \frac{1+x\mathrm{i}}{\sqrt{h(x)}},
\]
so that
\[
F(\theta) = \bigl| z_a(\theta+\tfrac{\pi}{2}) \bigr|^2 + \bigl| z_c(\theta) \bigr|^2
= (1+x^2)\bigl( \tfrac{1}{g(x)}+\tfrac{1}{h(x)}\bigr)
= \frac{(x^2+1)\bigl(g(x)+h(x)\bigr)}{g(x)h(x)}
\]
Consequently, $F(\theta)=1$ if and only if~\eqref{eq: degree four} holds, and for any such $x$, $f$ is a conference signal that satisfies $\hat{f}=\sqrt{n} \overline{f}$ when $a,b,c,d$ are given by~\eqref{eq: abcd}.
\end{proof}

\begin{remark}
While one might hope for a generalization of Theorem~\ref{thm: n = 3p} that constructs conference signals of order $n=pq$ for some other primes $p,q$, this sadly does not seem to work.
In the proof above, we used the fact that $q=3$ when we showed that $F(\theta)>1$; for any prime $q>3$ the ellipse $E_c$ would instead have major semiaxis length $\tfrac{1}{\sqrt{q}-1}<1$.
Furthermore, the authors asked ChatGPT~6 Astra Pro about such a generalization, and it reported that $F(\theta)<1$ for all $\theta$ whenever $q>3$.
\end{remark}

\section{Explicit construction (prime order)}
\label{sec: explicit prime}

In this section, which contains the second of our main results, we give another explicit infinite family of conference signals, now with prime order~$p$.
Our construction works whenever $p=1\bmod 6$ and $p \geq 13$.
None of these conference signals are multiples of multiplicative characters.
To the authors' knowledge, these are the first explicit counterexamples for Cohn's problem over prime fields.

Throughout this section, we use \textbf{Gauss and Jacobi sums}~\cite{BerndtEW:98} extensively, with the following notation: for any prime $p$ and any multiplicative characters $\psi,\phi \in \mathbb{C}^{\mathbb{Z}_p}$, we define
\begin{equation}
\label{eq: Gauss and Jacobi sums}
G(\psi):=\widehat{\psi}(1)
=\sum_{t\in\mathbb{Z}_p} \psi(t) \omega_p^{-t},
\qquad
J(\psi,\phi):=(\psi*\phi)(1)
=\sum_{t\in\mathbb{Z}_p}\psi(t)\phi(1-t).
\end{equation}
Famously, $|G(\psi)|=\sqrt{p}$ and $\widehat{\psi}=G(\psi)\overline{\psi}$ when $\psi$ is nontrivial.
When all three of $\psi,\phi,\psi\phi$ are nontrivial, this implies
\begin{equation}
\label{eq: Jacobi Gauss sum fraction}
J(\psi,\phi)= \frac{G(\psi) G(\phi)}{G(\psi \phi)}
\end{equation}
since
\[
\widehat{\psi * \phi}
= G(\psi)G(\phi)\overline{\psi \phi}
= \frac{G(\psi)G(\phi)}{G(\psi\phi)}\widehat{\psi\phi}.
\]
Note that our definition of $G(\psi)$ differs from the one in~\cite{BerndtEW:98} by a factor of $\psi(-1)\in \{ \pm 1\}$, which cancels in~\eqref{eq: Jacobi Gauss sum fraction}.

In the theorem below, we express members of $\mathbb{C}^{\mathbb{Z}_3}$ as sequences like $\{x_j\}_{j\in \mathbb{Z}_3}$ instead of functions like $j \mapsto x(j)$ to cut down on parentheses.

\begin{theorem}
\label{thm: prime construction}
Let $p \geq 13$ be a prime with $p = 1 \bmod 6$.
Fix a primitive element (multiplicative generator) $g \in \mathbb{Z}_p^\times$, and let $\varphi$ be the unique multiplicative character on $\mathbb{Z}_p$ with $\varphi(g)=\omega_{p-1}$.
Next, choose any integer $r$ with $1 \leq r < m:=\tfrac{p-1}{3}$, $r \neq \tfrac{m}{2}$.
To produce a conference signal $f$ on $\mathbb{Z}_p$ that satisfies $D_{g^3} f = \omega_m^{-r} f$, first define $z \in \mathbb{C}^{\mathbb{Z}_3}$ by
\begin{equation}
\label{eq: zj}
z_j = \frac{G(\varphi^{r+(j-1)m})}{G(\varphi^{r+jm})}
=\frac{J(\varphi^{r+(j-1)m},\varphi^m)}{G(\varphi^m)}
\qquad
\forall j\in \mathbb{Z}_3.
\end{equation}
Then there is exactly one choice of $\varepsilon \in \{\pm 1\}$ for which
\[
y_j:=\varepsilon(1+z_j)
\operatorname{Im}\!\left((1+z_{j-1})\overline{(1+z_{j+1})}\right)
\qquad
\forall j\in \mathbb{Z}_3
\]
satisfy $y_0y_1y_2>0$.
For this choice of $\varepsilon$, define $x\in\mathbb C^{\mathbb Z_3}$ by
\[
\widehat{x}_0=\left|\frac{y_0 y_1}{y_2}\right|^{1/2},
\quad
\widehat{x}_1=\left|\frac{y_2}{y_0 y_1}\right|^{1/2} y_1,
\quad
\widehat{x}_2=\left|\frac{y_2}{y_0 y_1}\right|^{1/2} \overline{y_0}.
\]
Then the function $f:\mathbb Z_p\to\mathbb C$ given by
\[
f(0)=0,
\qquad
f(g^{j+3\ell})=x_j\omega_{p-1}^{r(j+3\ell)}
\qquad
\forall j \in \{0,1,2\},\enspace \forall \ell \in \{0,\ldots,m-1\}
\]
is a multiple of a conference signal.
Furthermore, $f$ is not a multiple of a multiplicative character, and it satisfies
\[
\widehat f=G(\varphi^r)\overline{f},
\qquad
D_{g^3}f=\omega_m^{-r}f,
\qquad
f(-t)=(-1)^r f(t)
\qquad
\forall
t \in \mathbb{Z}_p,
\]
so that $f$ is odd when $r$ is odd and even when $r$ is even.
\end{theorem}

The hypothesis $p \geq 13$ above ensures that $m \geq 4$, so an appropriate choice of $r$ exists.
We develop our argument for Theorem~\ref{thm: prime construction} over the course of this section and give a proof in Section~\ref{subsec: irrationality and distinctness}.

\subsection{A circulant matrix for the DFT}
\label{subsec: circulant DFT}
Our argument uses a remarkable fact about the DFT of prime order: when restricted to the $\lambda$-eigenspace of any dilation, its matrix is circulant for a certain choice of bases, provided $\lambda \neq 1$.

This part of our argument works in greater generality than Theorem~\ref{thm: prime construction}, so we use a broader setup in this subsection.
Let $p$ be any odd prime, and fix a primitive element $g \in \mathbb{Z}_p^\times$.
To describe the relevant dilation eigenspaces, first observe that whenever two elements $g^j,g^k \in \mathbb{Z}_p^\times$ generate identical subgroups  of $\mathbb{Z}_p^\times$, their corresponding dilations $D_{g^j},D_{g^k}$ have identical sets of eigenspaces and identical $1$-eigenspaces.
Hence every $\lambda$-eigenspace of a dilation with $\lambda \neq 1$ occurs as such an eigenspace of $D_{g^d}$ for some $d \mid (p-1)$.
To fix our attention on one such eigenspace, fix $d \mid (p-1)$ with $d \neq p-1$, fix an integer $r$ with $1 \leq r < m:=\tfrac{p-1}{d}$, and consider
\[
V:=\bigl\{ f \in \mathbb{C}^{\mathbb{Z}_p} : D_{g^d} f = \omega_m^{-r} f \bigr\}.
\]
Here, every $f\in V$ satisfies $f(0)=0$ since $\omega_m^r \neq 1$.
Furthermore, any $f \in V$ is determined by $f(g^j)$ for $j \in \mathbb{Z}_d$ since each element of $\mathbb{Z}_p^\times$ can be written uniquely as $g^{j+d\ell}$ for $j\in \mathbb{Z}_d=\{0,\ldots,d-1\}$ and $\ell\in \mathbb{Z}_m=\{0,\ldots,m-1\}$, and since repeated dilation gives $f(g^{j+d\ell})=\omega_m^{r\ell} f(g^j)=\omega_{p-1}^{rd\ell} f(g^j)$.
In particular, $V$ has orthonormal basis $\{v_j\}_{j\in \mathbb{Z}_d}$ given by
\[
v_j(t) = \left\{
\begin{array}{ll}
\tfrac{1}{\sqrt{m}}\omega_{p-1}^{r(j+d\ell)} & \text{if $t=g^{j+d\ell}$ for some $\ell \in \mathbb{Z}_m$} \\
0 & \text{otherwise}
\end{array} \right\}
\qquad
\forall j \in \mathbb{Z}_d=\{0,\ldots,d-1\},
\enspace
\forall t \in \mathbb{Z}_p,
\]
where members of $V$ expand as
\begin{equation}
\label{eq: expand in V}
f = \sqrt{m}\sum_{j\in \mathbb{Z}_d} f(g^j) \omega_{p-1}^{-rj} v_j
\qquad
\forall f \in V.
\end{equation}
As a consequence of the relation $\widehat{D_{g^d} f} = D_{g^{-d}} \hat{f}$, the Fourier transform over $\mathbb{Z}_p$ maps $V$ onto
\[
\bigl\{ f \in \mathbb{C}^{\mathbb{Z}_p} : D_{g^d} f = \omega_m^r f \bigr\}
= \bigl\{ \overline{f} : f \in V \bigr\}
=: \overline{V}.
\]
We equip $\overline{V}$ with the orthonormal basis $\{w_k\}_{k\in \mathbb{Z}_d}$ given by $w_k = \overline{v_{-k}}$, where
\begin{equation}
\label{eq: expand in Vbar}
f = \sqrt{m} \sum_{k\in \mathbb{Z}_d} f(g^{-k}) \omega_{p-1}^{-rk} w_k
\qquad
\forall f \in \overline{V},
\end{equation}
as seen by taking a complex conjugate in~\eqref{eq: expand in V}.

\begin{example}
To illustrate our notation, take $p=5$ and $d=2$ with $g=2$.
In $\mathbb{Z}_5$ we have $(1,2,3,4)=(g^0,g^1,g^3,g^2)$, so our bases are given by
\[
v_0 = w_0 = \tfrac{1}{\sqrt{2}}
\begin{bmatrix}0 & 1 & 0 & 0 & -1\end{bmatrix}^{\top},
\qquad
v_1 = \overline{w_1} = \tfrac{1}{\sqrt{2}}
\begin{bmatrix}0 & 0 & \mathrm{i} & -\mathrm{i} & 0\end{bmatrix}^{\top}
.
\]
In this example, $V = \overline{V}$ as sets, but their bases are different.
\end{example}

As in Theorem~\ref{thm: prime construction}, we write $\varphi$ for the unique multiplicative character on $\mathbb{Z}_p$ with $\varphi(g) = \omega_{p-1}$, and we express members of $\mathbb{C}^{\mathbb{Z}_d}$ as sequences rather than functions.

\begin{theorem}
\label{thm: circulant DFT}
In the notation above, the DFT over $\mathbb{Z}_p$ restricts to an isomorphism $\mathcal{F}|_V \colon V \to \overline{V}$ whose matrix is a circulant unitary.
Specifically, define
\begin{equation}
\label{eq: tau}
\tau_j := \sum_{\ell=0}^{m-1} \omega_{p-1}^{r(-j+d\ell)} \omega_p^{-g^{-j+d\ell}}
\qquad
\forall j \in \mathbb{Z},
\end{equation}
where $\tau_j = \tau_{j+d}$; treating $\tau$ as a member of $\mathbb{C}^{\mathbb{Z}_d}$, its Fourier transform is given by
\[
\widehat{\tau}_j = G(\varphi^{r+jm})
\qquad
\forall j \in \mathbb{Z}_d.
\]
Then $\tfrac{1}{\sqrt{p}}C_\tau$ is the matrix of $\mathcal{F}|_V \colon V \to \overline{V}$.

Furthermore, the following hold for any $f = \sum_{j \in \mathbb{Z}_d} x_j v_j \in V$ with coordinate vector $x \in \mathbb{C}^{\mathbb{Z}_d}$:
\begin{itemize}
\item[(a)]
$f$ is a nonzero multiple of a conference signal if and only if $x$ and $\tau * x$ are both flat,
\item[(b)]
$f$ is a multiple of a multiplicative character if and only if $\widehat{x}$ is a multiple of $\delta_j$ for some~$j \in \mathbb{Z}_d$,
\item[(c)]
for any $\lambda \in \mathbb{C}$,
\[
\hat{f} = \lambda \overline{f}
\quad\iff\quad
\tau * x = \lambda x^*
\quad\iff\quad
G(\varphi^{r+jm}) \widehat{x}_j = \lambda \overline{\widehat{x}_j} \quad \forall j \in \mathbb{Z}_d.
\]
\end{itemize}
\end{theorem}

In~\eqref{eq: tau}, $\tau_j$ is a kind of ``twisted Gauss period'', where a \textit{Gauss period}~\cite{BerndtEW:98} would have $r=0$.

\begin{proof}[Proof of Theorem~\ref{thm: circulant DFT}]
For periodicity, reindex~\eqref{eq: tau} with $k=\ell+1$ and peel off the top term:
\[
\tau_j = \sum_{k=1}^m \omega_{p-1}^{r(-j-d+dk)} \omega_p^{-g^{-j-d+dk}}
= \omega_{p-1}^{r(-j-d+dm)} \omega_p^{-g^{-j-d+dm}} + \sum_{k=1}^{m-1} \omega_{p-1}^{r(-j-d+dk)} \omega_p^{-g^{-j-d+dk}}.
\]
This will equal $\tau_{j+d}$ once we show the peeled-off term coincides with the missing $k=0$ term on the far right.
Indeed, we can use the fact that $dm=p-1$ and $g^{p-1}=1$ to write
\[
\omega_{p-1}^{r(-j-d+dm)} \omega_p^{-g^{-j-d+dm}}
= \omega_{p-1}^{r(-j-d)}\omega_{p-1}^{r(p-1)} \omega_p^{-g^{-j-d}\cdot g^{p-1}}
=\omega_{p-1}^{r(-j-d)} \omega_p^{-g^{-j-d}}.
\]
Thus, $\tau_j = \tau_{j+d}$.
We compute $\widehat{\tau}$ further below.

For the matrix of $\mathcal{F}|_V$, fix $j \in \mathbb{Z}_d$ and expand $\sqrt{p}\mathcal{F} v_j = \widehat{v}_j$ with~\eqref{eq: expand in Vbar}:
\[
\widehat{v}_j
= \sqrt{m} \sum_{k\in \mathbb{Z}_d} \widehat{v}_j(g^{-k}) \omega_{p-1}^{-rk} w_k.
\]
Here, the coefficient on $w_k$ is
\[
\sqrt{m} \omega_{p-1}^{-rk} \widehat{v}_j(g^{-k})
= \sqrt{m} \omega_{p-1}^{-rk} \sum_{t \in \mathbb{Z}_p} v_j(t) \omega_p^{-tg^{-k}}
= \omega_{p-1}^{-rk} \sum_{\ell=0}^{m-1} \omega_{p-1}^{r(j+d\ell)}\omega_p^{-g^{j+d\ell}\cdot g^{-k}}
= \tau_{k-j}.
\]
This matches the $(k,j)$-entry of $C_\tau$ since $\tau$ is $d$-periodic.
Consequently, $\widehat{v}_j = \sum_{k\in \mathbb{Z}_d} [C_\tau]_{k,j} w_k$ and $\sqrt{p}\mathcal{F}|_V$ has matrix $C_\tau$.

For (a), recall that $f$ and $\hat{f}$ both vanish at $0$ since they belong to $V$ and $\overline{V}$, respectively.
Hence, $f$ is a nonzero multiple of a conference signal if and only if $f$ and $\hat{f}$ are both flat on $\mathbb{Z}_p\setminus\{0\}$.
To interpret this in terms of coordinate vectors, observe that the basis vectors for each of $V$ and $\overline{V}$ have modulus $\tfrac{1}{\sqrt{m}}$ on their supports, which partition $\mathbb{Z}_p\setminus\{0\}$.
Consequently, $f$ and $\hat{f}$ are flat on $\mathbb{Z}_p\setminus\{0\}$ if and only if their coordinate vectors $x$ and $C_\tau x = \tau * x$ are flat.

For (b), we first identify the multiplicative characters in $V$.
Each multiplicative character can be written as $\varphi^s$ for some $s \in \mathbb{Z}_{p-1}$, and since
\[
D_{g^d}\varphi^s(g^t)
=\varphi^s(g^{t-d})
= \omega_{p-1}^{st-sd}
=\omega_m^{-s} \omega_{p-1}^{st},
\qquad
\omega_m^{-r} \varphi^s(g^t)= \omega_m^{-r} \omega_{p-1}^{st}
\qquad
\forall t \in \mathbb{Z}_{p-1},
\]
$\varphi^s$ belongs to $V$ if and only if $s=r+jm$ for some $j \in \mathbb{Z}_d$.
Next, we find the coordinate vector of $\varphi^{r+jm}$ by expanding with~\eqref{eq: expand in V}:
\[
\varphi^{r+jm}
= \sqrt{m} \sum_{k \in \mathbb{Z}_d} \varphi^{r+jm}(g^k) \omega_{p-1}^{-rk} v_k
= \sqrt{m} \sum_{k\in \mathbb{Z}_d} \omega_{p-1}^{k(r+jm)} \omega_{p-1}^{-rk} v_k
= \sqrt{m} \sum_{k\in \mathbb{Z}_d} \omega_d^{jk} v_k
\qquad
\forall j \in \mathbb{Z}_d.
\]
Thus, $\varphi^{r+jm}$ has coordinate vector $\{\sqrt{m} \omega_d^{jk} \}_{k \in \mathbb{Z}_d} = d \sqrt{m}\check{\delta}_j$, i.e., the vector whose Fourier transform over $\mathbb{Z}_d$ is $d \sqrt{m} \delta_j$.
This easily implies (b).
Furthermore, we can use it to show $\widehat{\tau}_j = G(\varphi^{r+jm})$.
Indeed,
\[
\widehat{\varphi^{r+jm}} = \sqrt{p} \mathcal{F} \varphi^{r+jm}
= d \sqrt{m} \sum_{k \in \mathbb{Z}_d} (\tau * \check{\delta}_j)_k w_k
= d \sqrt{m} \sum_{k \in \mathbb{Z}_d} \widehat{\tau}_j (\check{\delta}_j)_k w_k,
\]
where the last equality is easily verified in the Fourier domain;
evaluating at $1$ gives
\[
G(\varphi^{r+jm}) = \widehat{\varphi^{r+jm}}(1)
= d \sqrt{m} \widehat{\tau}_j \sum_{k \in \mathbb{Z}_d}  (\check{\delta}_j)_k \overline{v_{-k}}(g^0)
= \widehat{\tau}_j.
\]

Finally, for (c), we expand
\[
\hat{f} = \sum_{k\in \mathbb{Z}_d} (\tau * x)_k w_k,
\qquad
\lambda \overline{f} = \lambda \sum_{j\in \mathbb{Z}_d} \overline{x_j} \overline{v_j}
= \lambda \sum_{k\in \mathbb{Z}_d} \overline{x_{-k}} \overline{v_{-k}}
= \sum_{k\in \mathbb{Z}_d} (\lambda x^*)_k w_k.
\]
Comparing coefficients gives the first equivalence of~(c).
Taking the Fourier transform over $\mathbb{Z}_d$ gives the second equivalence, since $\widehat{\tau}_j=G(\varphi^{r+jm})$ and $\widehat{x^*}=\overline{\widehat{x}}$.
\end{proof}

\begin{example}
When $d=1$, Theorem~\ref{thm: circulant DFT}(b) implies every member of $V$ is a multiple of a multiplicative character.
Indeed, the only basis vector of $V$ is $v_0=\tfrac{1}{\sqrt{p-1}}\varphi^r$.
\end{example}

\begin{remark}
When $d=2$, one can show every conference signal in $V$ is a multiple of a multiplicative character.
To get an idea of why this holds, apply Theorem~\ref{thm: circulant DFT}(a) and consider the action of the $2 \times 2$ unitary $\tfrac{1}{\sqrt{p}} C_\tau$ on the Bloch sphere, as in Example~\ref{ex: Bloch flat} and Figure~\ref{fig: Bloch}.
The projective points corresponding to $\check{\delta}_0$ and $\check{\delta}_1$ are fixed by the action of $\tfrac{1}{\sqrt{p}} C_\tau$ since Fourier modes are eigenvectors of circulant matrices.
To show these are the only two points of intersection for the red and purple circles in Figure~\ref{fig: Bloch}, it suffices to show those circles are distinct for this unitary.
The only unitaries that preserve the red circle are monomial matrices, and our circulant matrix $C_\tau = \left[ \begin{smallmatrix*} \tau_0 & \tau_1 \\ \tau_1 & \tau_0 \end{smallmatrix*} \right]$ would be monomial only if one of $\tau_0,\tau_1$ vanished.
However, one can use ideas like in the proof of Lemma~\ref{lem: irrationality and distinctness}(a) below to show this does not happen because each of $\tau_0,\tau_1$ is irrational.
\end{remark}

\subsection{Biunimodular vectors for $3 \times 3$ circulant unitaries}
Our proof of the main result (Theorem~\ref{thm: prime construction}) applies Theorem~\ref{thm: circulant DFT} with $d=3$, which is the smallest value not ruled out above.
Here, Theorem~\ref{thm: circulant DFT}(a) calls for a vector $x\in \mathbb{C}^{\mathbb{Z}_3}$ that is \textbf{biunimodular} for the $3\times 3$ circulant unitary $\tfrac{1}{\sqrt{p}}C_\tau$, meaning both $x$ and $\tfrac{1}{\sqrt{p}}C_\tau x$ have all unimodular entries.
Any Fourier mode satisfies this condition as well as the con-eigenvector condition in Theorem~\ref{thm: circulant DFT}(c), but the resulting conference signal is a multiple of a multiplicative character by Theorem~\ref{thm: circulant DFT}(b).
In this subsection, we move beyond Fourier modes by presenting a general method to construct a biunimodular vector for any $3 \times 3$ circulant unitary with two mild properties, as in the theorem below; the resulting vector also satisfies the analogue of the con-eigenvector condition in Theorem~\ref{thm: circulant DFT}(c) and is not a multiple of a Fourier mode.
Theorem~\ref{thm: prime construction} implements this construction for $\tfrac{1}{\sqrt{p}}C_\tau$; we will check the ``two mild properties'' in Section~\ref{subsec: irrationality and distinctness}.

\begin{theorem}
\label{thm: 3x3 circulant unitaries}
Let
$C_u = \left[ \begin{smallmatrix*}
u_0 & u_2 & u_1 \\
u_1 & u_0 & u_2 \\
u_2 & u_1 & u_0
\end{smallmatrix*} \right]$
be a $3 \times 3$ circulant unitary, where $u \in \mathbb{C}^{\mathbb{Z}_3}$.
To produce a biunimodular vector for $C_u$, define $z_j := \widehat{u}_{j-1} \overline{\widehat{u}_j}$ for each $j \in \mathbb{Z}_3$,  and assume the following hold:
\begin{itemize}
\item[(i)]
$z_j \neq -1$ for each $j$,
\item[(ii)]
$z_0,z_1,z_2$ are pairwise distinct.
\end{itemize}
Then there is exactly one choice of $\varepsilon \in \{\pm 1\}$ for which
\begin{equation}
\label{eq: yj}
y_j:=\varepsilon(1+z_j)
\operatorname{Im}\!\left((1+z_{j-1})\overline{(1+z_{j+1})}\right)
\qquad
\forall j\in \mathbb{Z}_3
\end{equation}
satisfy $y_0y_1y_2>0$.
For this choice of $\varepsilon$, define $x\in\mathbb C^{\mathbb Z_3}$ by
\begin{equation}
\label{eq: xj}
\widehat{x}_0=\left|\frac{y_0 y_1}{y_2}\right|^{1/2},
\quad
\widehat{x}_1=\left|\frac{y_2}{y_0 y_1}\right|^{1/2} y_1,
\quad
\widehat{x}_2=\left|\frac{y_2}{y_0 y_1}\right|^{1/2} \overline{y_0}.
\end{equation}
Then $x$ and $u * x$ are both flat, with $u * x = \widehat{u}_0 x^*$.
Furthermore, $x$ is not a multiple of a Fourier mode, where $\widehat{x}_j \neq 0$ for every $j \in \mathbb{Z}_3$.
\end{theorem}

Relatedly, see~\cite{Fuhr:15} for a way to construct a biunimodular vector for any $3 \times 3$ unitary, not necessarily circulant.

\begin{remark}
\label{rem: y explained}
Before we prove the theorem, we give some context that will hopefully make the construction less mysterious.
Let $u$ be as in the theorem statement (so that each $\widehat{u}_j$ is unimodular), and consider the problem of finding $x \in \mathbb{C}^{\mathbb{Z}_3}$ such that $x$ is flat, $u*x=\lambda x^*$ for some (necessarily unimodular) $\lambda \in \mathbb{C}$, and $\widehat{x}_j \neq 0$ for every $j \in \mathbb{Z}_3$.
For simplicity, assume that last condition already holds, so that $x \neq 0$.
To move this problem to the frequency domain, observe that $x$ is flat if and only if $|x|^2 = x \cdot \overline{x}$ is constant, if and only if $\widehat{|x|^2}$ is a multiple of $\delta_0$.
Here, $(\widehat{|x|^2})_1$ and $(\widehat{|x|^2})_2$ are complex conjugates since $|x|^2$ is real-valued, so $\widehat{|x|^2}$ is a multiple of $\delta_0$ if and only if $0=(\widehat{|x|^2})_1=\tfrac{1}{3} (\widehat{x} * \widehat{x}^*)_1$.
By rescaling and opening the convolution, we find $x$ is flat if and only if
\[
\widehat{x}_0 \overline{\widehat{x}_2} 
+ \widehat{x}_1 \overline{\widehat{x}_0} 
+ \widehat{x}_2 \overline{\widehat{x}_1}
= 0.
\]
Likewise, $u*x=\lambda x^*$ if and only if $\widehat{u}_j \widehat{x}_j = \lambda \overline{\widehat{x}_j}$ for each $j \in \mathbb{Z}_3$.

Faced with the nonlinear equation above, we take inspiration from Bj\"{o}rck~\cite{Bjorck:87} and introduce $y_j := \widehat{x}_j \overline{\widehat{x}_{j-1}}$ and $z_j := \widehat{u}_{j-1} \overline{\widehat{u}_j}$ for each $j \in \mathbb{Z}_3$.
Then $y_0 y_1 y_2 = \prod_j |\widehat{x}_j|^2> 0$, and $x$ is flat if and only if $y_0+y_1+y_2=0$ (phew!).
Furthermore, $u*x=\lambda x^*$ only if $y_j = z_j \overline{y_j}$ for each $j$, since
\[
\frac{y_j}{\overline{y_j}}
= \frac{ \widehat{x}_j \overline{\widehat{x}_{j-1}} }{ \overline{\widehat{x}_j} \widehat{x}_{j-1} }
= \frac{\overline{\widehat{x}_{j-1}}/\widehat{x}_{j-1}}{\overline{\widehat{x}_j}/\widehat{x}_j}
= \frac{\widehat{u}_{j-1}/\lambda}{\widehat{u}_j/\lambda}
= z_j.
\]
The condition $y_j = z_j \overline{y_j}$ is also linear in the real and imaginary parts of $y_j$; to connect it with our choice of $y$ in~\eqref{eq: yj}, we transform it further.
One way to achieve this condition is to set $y_j = (1+z_j)c_j$ for $c_j \in \mathbb{R}$ as this implies
\[
z_j \overline{y_j} = z_j(1+\overline{z_j})c_j =(z_j+1)c_j=y_j;
\]
conversely, when $z_j \neq -1$ this is the only way to achieve the condition $y_j = z_j \overline{y_j}$ as this implies
\[
\overline{\left( \frac{y_j}{1+z_j} \right)} 
= \frac{\overline{y_j}}{1+\overline{z_j}} \cdot \frac{z_j}{z_j}
= \frac{y_j}{z_j+1}.
\]
Recalling the linear constraint $y_0+y_1+y_2=0$, we can narrow the possibilities for $y$ by solving
\[
\left\{ \begin{array}{l}
(1+z_0)c_0+(1+z_1)c_1+(1+z_2)c_2=0 \\
(1+\overline{z_0})c_0+(1+\overline{z_1})c_1+(1+\overline{z_2})c_2=0
\end{array} \right.
\]
for $c_0,c_1,c_2 \in \mathbb{R}$, where adding and subtracting equations shows the system is equivalent to
\[
\left\{ \begin{array}{l}
\operatorname{Re}(1+z_0)c_0+\operatorname{Re}(1+z_1)c_1+\operatorname{Re}(1+z_2)c_2=0 \\
\operatorname{Im}(1+z_0)c_0+\operatorname{Im}(1+z_1)c_1+\operatorname{Im}(1+z_2)c_2=0.
\end{array} \right.
\]
Under the hypotheses of our theorem, this system turns out to have a $1$-dimensional solution space, spanned by $c_j=
\operatorname{Im}\!\left((1+z_{j-1})\overline{(1+z_{j+1})}\right)$.
This explains our choice of $y$ in~\eqref{eq: yj}, while $x$ is designed to satisfy $\widehat{x}_j \overline{\widehat{x}_{j-1}}=y_j$ in~\eqref{eq: xj}.
\end{remark}

\begin{proof}[Proof of Theorem~\ref{thm: 3x3 circulant unitaries}]
First, we verify the claim about $y_0y_1y_2$.
Leaving $\varepsilon \in \{\pm 1\}$ undetermined for the moment, observe that each $y_j = (1+z_j)c_j$ for some $c_j \in \mathbb{R}$; as explained above, this implies $y_j=z_j \overline{y_j}$.
Then $y_0y_1y_2= z_0z_1z_2\overline{y_0y_1y_2}=\overline{y_0y_1y_2}$ is real.
To show some choice of $\varepsilon$ makes $y_0y_1y_2>0$, it now suffices to show each $y_j\neq 0$, so that $y_0 y_1 y_2 \neq 0$ and opposite choices of $\varepsilon$ give opposite signs for $y_0y_1y_2$.
Considering (i) and the definition of $y$, it further suffices to show $(1+z_j)\overline{(1+z_k)} \not \in \mathbb{R}$ whenever $j \neq k$.
Suppose on the contrary that $(1+z_j)\overline{(1+z_k)}=a \in \mathbb{R}$ for some $j \neq k$.
Multiplying both sides by $\tfrac{1+z_k}{|1+z_k|^2}$, we find $1+z_j=b(1+z_k)$ for some $b \in \mathbb{R}$, so $z_j$ falls on the line $z=b(1+z_k)-1$, $b \in \mathbb{R}$.
However, any line in the complex plane intersects the unit circle in at most two points, and our line already intersects it at $z=-1$ (when $b=0$) and $z=z_k$ (when $b=1$).
Hence $z_j\in\{-1,z_k\}$, contrary to~(i) and~(ii).
This proves the claim.

Next, we use the strategy from Remark~\ref{rem: y explained} to show $x$ is flat.
Indeed, each $\widehat{x}_j \neq 0$ since $y_0,y_1,y_2\neq 0$, and
\[
\widehat{x}_0\overline{\widehat{x}_2}=y_0,
\qquad
\widehat{x}_1\overline{\widehat{x}_0}=y_1,
\qquad
\widehat{x}_2\overline{\widehat{x}_1}
= \left|\frac{y_2}{y_0y_1}\right|\overline{y_0y_1}
= \frac{|y_2|^2}{\overline{y_2}} \cdot \frac{\overline{y_0 y_1 y_2}}{|y_0y_1y_2|}
= y_2,
\]
where the last step uses the fact that $y_0y_1y_2>0$.
To show $x$ is flat, it now suffices to show $y_0+y_1+y_2=0$.
Indeed, write $1+z_k=a_k+b_k\mathrm{i}$ with $a_k,b_k\in \mathbb{R}$ for each $k$.
Then 
\[
\operatorname{Im}\!\left((1+z_k)\overline{(1+z_\ell)}\right)=b_ka_\ell-a_kb_\ell
\qquad
\forall k,\ell \in \mathbb{Z}_3,
\]
so for each $j\in \mathbb{Z}_3$,
\[
\varepsilon y_j
= (a_j + b_j \mathrm{i})(b_{j-1} a_{j+1} - a_{j-1} b_{j+1})
= (a_j b_{j-1} a_{j+1} - a_j a_{j-1} b_{j+1})
+ (b_j b_{j-1} a_{j+1} - b_j a_{j-1} b_{j+1})\mathrm{i}.
\]
When we sum this over all $j\in \mathbb{Z}_3$, the terms cancel in pairs.
Thus, $x$ is flat.

Finally, to show $u*x=\widehat{u}_0 x^*$, we verify that $\widehat{u}_j \widehat{x}_j = \widehat{u}_0 \overline{\widehat{x}_j}$ for each $j \in \mathbb{Z}_3$.
Indeed, the condition $y_j = z_j \overline{y_j}$ says that $\widehat{x}_j \overline{\widehat{x}_{j-1}} = \widehat{u}_{j-1} \overline{\widehat{u}_j} \overline{\widehat{x}_j} \widehat{x}_{j-1}$, so
\[
\frac{ \widehat{u}_j\widehat{x}_j}{\overline{\widehat{x}_j}} = \frac{\widehat{u}_{j-1}\widehat{x}_{j-1}}{\overline{\widehat{x}_{j-1}}}
\qquad
\forall j \in \mathbb{Z}_3.
\]
In other words, the value of $\frac{ \widehat{u}_j \widehat{x}_j}{\overline{\widehat{x}_j}}$ is independent of $j$.
Taking $j=0$, we see the common value is $\widehat{u}_0$ since $\widehat{x}_0$ is real.
Thus, $\widehat{u}_j\widehat{x}_j = \widehat{u}_0 \overline{\widehat{x}_j}$ for every $j \in \mathbb{Z}_3$.
\end{proof}

\subsection{Irrationality and distinctness of Gauss and Jacobi sums}
\label{subsec: irrationality and distinctness}

To prove our main result, we now show $\tfrac{1}{\sqrt{p}} C_\tau$ satisfies the ``two mild properties'' from Theorem~\ref{thm: 3x3 circulant unitaries}.

\begin{lemma}
\label{lem: irrationality and distinctness}
Let $p \geq 13$ be a prime with $p = 1 \bmod 6$.
Fix a primitive element $g \in \mathbb{Z}_p^\times$, and let $\varphi$ be the unique multiplicative character on $\mathbb{Z}_p$ with $\varphi(g)=\omega_{p-1}$.
Next, choose any integer $r$ with $1 \leq r < m:=\tfrac{p-1}{3}$, $r \neq \tfrac{m}{2}$, and define $z \in \mathbb{C}^{\mathbb{Z}_3}$ by
\begin{equation}
\label{eq: zj again}
z_j = \frac{G(\varphi^{r+(j-1)m})}{G(\varphi^{r+jm})}
=\frac{J(\varphi^{r+(j-1)m},\varphi^m)}{G(\varphi^m)}
\qquad
\forall j\in \mathbb{Z}_3.
\end{equation}
Then the following hold:
\begin{itemize}
\item[(a)]
$z_j$ is irrational for each $j$,
\item[(b)]
$z_0,z_1,z_2$ are pairwise distinct.
\end{itemize}
\end{lemma}

We will prove (a) and (b) separately below.
In the meantime, we temporarily assume the lemma and put all the pieces together to prove the main result.

\begin{proof}[Proof of Theorem~\ref{thm: prime construction}, assuming Lemma~\ref{lem: irrationality and distinctness}]

To find an appropriate coordinate vector for Theorem~\ref{thm: circulant DFT}, first apply Theorem~\ref{thm: 3x3 circulant unitaries} with $u=\tfrac{1}{\sqrt{p}} \tau$, where $C_u$ is unitary and $\widehat{u}_j=\tfrac{1}{\sqrt{p}}G(\varphi^{r+jm})$ for each $j\in \mathbb{Z}_3$
by Theorem~\ref{thm: circulant DFT}. 
In the statement of Theorem~\ref{thm: 3x3 circulant unitaries},
\[
z_j = \widehat{u}_{j-1} \overline{\widehat{u}_j}
= \frac{G(\varphi^{r+(j-1)m}) \overline{G(\varphi^{r+jm})}}{p}
= \frac{G(\varphi^{r+(j-1)m})}{G(\varphi^{r+jm})}
=\frac{J(\varphi^{r+(j-1)m},\varphi^m)}{G(\varphi^m)}
\qquad
\forall j\in \mathbb{Z}_3,
\]
where the last two steps apply $|G(\varphi^{r+jm})|^2=p$ and~\eqref{eq: Jacobi Gauss sum fraction}; this matches the value of $z_j$ in~\eqref{eq: zj}.
By Lemma~\ref{lem: irrationality and distinctness}, the ``two mild properties'' hold in Theorem~\ref{thm: 3x3 circulant unitaries}, so we may apply that theorem to produce $y$ and $x$ exactly as stated in Theorem~\ref{thm: prime construction}.
By Theorem~\ref{thm: 3x3 circulant unitaries}, $x$ and $\tau * x$ are both flat, $x$ is not a multiple of a Fourier mode, and $\tau * x = \sqrt{p}\widehat{u}_0 x^* = G(\varphi^r) x^*$.
The same properties hold for $\sqrt{m} x$, which we take as a coordinate vector for $f := \sum_{j\in \mathbb{Z}_3} \sqrt{m} x_j v_j \in V$, where $V$ and $\{v_j\}_{j\in \mathbb{Z}_3}$ are as defined in Section~\ref{subsec: circulant DFT} for the case $d=3$.
Here, $f(0)=0$ and $f(g^{j+3\ell})=\sqrt{m}x_j v_j(g^{j+3\ell})=x_j \omega_{p-1}^{r(j+3\ell)}$ for all $j \in \mathbb{Z}_3$ and $\ell \in \mathbb{Z}_m$, matching the definition of $f$ in Theorem~\ref{thm: prime construction}.
By Theorem~\ref{thm: circulant DFT}, $f$ is a multiple of a conference signal, it is not a multiple of a multiplicative character, and $\hat{f} = G(\varphi^r) \overline{f}$.
Furthermore, $D_{g^3} f = \omega_m^{-r} f$ since $f\in V$.
Finally, to see that $f(-t)=(-1)^r f(t)$, observe that $m=\tfrac{p-1}{3}$ is even since $p = 1 \bmod 6$, and $(g^3)^{m/2}=g^{(p-1)/2}=-1$; thus,
\[
f(-t)=D_{-1}f(t)
=(D_{g^3}^{m/2}f)(t)
=\omega_m^{-rm/2} f(t)
=(-1)^r f(t)
\qquad
\forall t \in \mathbb{Z}_p.
\qedhere
\]
\end{proof}

\subsubsection{Irrationality}

We will prove each $z_j$ is irrational in Lemma~\ref{lem: irrationality and distinctness} by showing some field automorphism moves it.
This will take only basic Galois theory, as below.

\begin{proposition}[{\cite[Theorem~26 on p596]{DummitF:04}}]
\label{prop: Galois}
For any $n \geq 2$, the Galois group of $\mathbb{Q}(\omega_n)$ is isomorphic to $\mathbb{Z}_n^\times$
by mapping each $k \in \mathbb{Z}_n^\times$ to the unique field automorphism $\sigma_k$ with $\sigma_k(\omega_n)=\omega_n^k$.
\end{proposition}

\begin{proof}[Proof of Lemma~\ref{lem: irrationality and distinctness}(a)]
We first identify an appropriate field.
Any multiplicative character $\psi$ takes values in $\mathbb{Q}(\omega_{p-1})$ since $\psi(0)=0$ and $\psi(g^j)^{p-1}=\psi(g^{j(p-1)})=\psi(1)=1$ for all $j \in \mathbb{Z}_{p-1}$.
Consequently, 
\begin{equation}
\label{eq: G psi}
G(\psi)=\widehat{\psi}(1)=\sum_{t \in \mathbb{Z}_p} \psi(t) \omega_p^{-t}
\end{equation}
lies in $\mathbb{Q}(\omega_p,\omega_{p-1})$.
Here, $\mathbb{Q}(\omega_p,\omega_{p-1})=\mathbb{Q}(\omega_{p(p-1)})$ since 
\[
\omega_{p-1}=\omega_{p(p-1)}^p,
\qquad
\omega_p=\omega_{p(p-1)}^{p-1},
\qquad\omega_{p(p-1)}=\omega_{p(p-1)}^{p-(p-1)}=\omega_{p-1}\omega_p^{-1}.
\]

Next, we identify an appropriate automorphism of $\mathbb{Q}(\omega_{p(p-1)})$.
By the Chinese Remainder Theorem, $\mathbb{Z}_{p(p-1)}^\times \cong \mathbb{Z}_p^\times \times \mathbb{Z}_{p-1}^\times$ via $k \mapsto \bigl(k \bmod p, k \bmod (p-1)\bigr)$, so there exists $k \in \mathbb{Z}_{p(p-1)}^\times$ with $k =g \bmod p$ and $k =1 \bmod (p-1)$.
Then the automorphism $\sigma_k$ from Proposition~\ref{prop: Galois} satisfies $\sigma_k(\omega_p)=\omega_p^k=\omega_p^g$ and $\sigma_k(\omega_{p-1})=\omega_{p-1}^k=\omega_{p-1}$; in particular, $\sigma_k$ fixes $\mathbb{Q}(\omega_{p-1})$.
For any multiplicative character $\psi$, applying $\sigma_k$ to both sides of~\eqref{eq: G psi} and then reindexing with $s=gt$ gives
\[
\sigma_k\bigl(G(\psi)\bigr)
= \sum_{t \in \mathbb{Z}_p} \sigma_k\bigl(\psi(t)\bigr) \sigma_k(\omega_p)^{-t}
= \sum_{t \in \mathbb{Z}_p} \psi(t) \omega_p^{-gt}
= \sum_{s\in \mathbb{Z}_p} \psi(g^{-1}s) \omega_p^{-s}
= \overline{\psi(g)} G(\psi).
\]
Thus, for any $j\in \mathbb{Z}_3$, $\sigma_k$ moves
\[
z_j = \frac{G(\varphi^{r+(j-1)m})}{G(\varphi^{r+jm})} \in \mathbb{Q}(\omega_p,\omega_{p-1})
\]
to
\[
\sigma_k(z_j)
= \frac{\overline{\omega_{p-1}^{r+(j-1)m}} G(\varphi^{r+(j-1)m})}{\overline{\omega_{p-1}^{r+jm}}G(\varphi^{r+jm})}
= \overline{\omega_{p-1}^{-m}} z_j
= \omega_3 z_j.
\]
Since $z_j \neq 0$, we conclude that $\sigma_k(z_j) \neq z_j$, so $z_j$ is irrational.
\end{proof}

\subsubsection{Distinctness}

We will show $z_0,z_1,z_2$ are distinct by showing the corresponding Jacobi sums in~\eqref{eq: zj} are distinct since they generate distinct ideals in a certain ring.
This will take some number theory.
For background, we now summarize aspects of the discussion on pages 65--66 of~\cite{BerndtEW:98}, which itself references~\cite{Washington:97}.
Let notation be as in Lemma~\ref{lem: irrationality and distinctness}.
For any choice of integers $a,b$, the Jacobi sum $J(\varphi^a,\varphi^b)$ lies in $\mathbb{Q}(\omega_{p-1})$, as seen from~\eqref{eq: Gauss and Jacobi sums}.
In fact, it lies in the ring $\mathcal{O}$ of algebraic integers in $\mathbb{Q}(\omega_{p-1})$.
It is well known that $\mathcal{O}$ is a Dedekind domain, so every nonzero ideal in $\mathcal{O}$ factors uniquely as a product of prime ideals.
For the case of an ideal generated by a nontrivial Jacobi sum $J(\varphi^a,\varphi^b)$, this factorization is known, as summarized in the proposition below.
To understand its notation, first fix a prime ideal $Q \triangleleft \mathcal{O}$ that divides $p \mathcal{O}$.
Then $\mathcal{O}/Q \cong \mathbb{Z}_p$, and there is a unique choice of $\ell \in \mathbb{Z}_{p-1}^\times$ for which $g+Q=\omega_{p-1}^\ell+Q$.
Next, we modify $Q$ to take $\ell = 1$:
using the notation of Proposition~\ref{prop: Galois} for automorphisms of $\mathbb{Q}(\omega_{p-1})$, put $P:=\sigma_\ell^{-1}(Q) \triangleleft \mathcal{O}$, so that
\[
g-\omega_{p-1}=\sigma_\ell^{-1}(g-\omega_{p-1}^\ell) \in \sigma_\ell^{-1}(Q)=P
\]
and $g+P=\omega_{p-1}+P$.
Then it holds that
\begin{equation}
\label{eq: factor pO}
p\mathcal{O} = \prod_{k \in \mathbb{Z}_{p-1}^\times} \sigma_k^{-1}(P),
\end{equation}
and the factors in this product are distinct prime ideals.
Finally, given $a \in \mathbb{Z}$, we write $\pi(a)$ for the corresponding member of $\mathbb{Z}_{p-1}=\{0,\ldots,p-2\}$ with $\pi(a) = a \bmod (p-1)$.

\begin{proposition}[{\cite[Theorem~2.1.14 on p66]{BerndtEW:98}}]
\label{prop: ideals}
Choose any $a,b\in \mathbb{Z}_{p-1}$ with $a,b,a+b\neq0$.
Then, in the notation above,
\[
J(\varphi^a,\varphi^b)\mathcal{O}
= \prod_{k \in S}
\sigma_k^{-1}(P),
\qquad
S:=\{k \in \mathbb{Z}_{p-1}^\times : \pi(k a)+\pi(k b)<p-1\}.
\]
\end{proposition}

\begin{proof}[Proof of Lemma~\ref{lem: irrationality and distinctness}(b)]
By the formula for $z_{j+1}$ in~\eqref{eq: zj again}, it suffices to show the three Jacobi sums $J(\varphi^{r+jm},\varphi^m)$, $j\in \mathbb{Z}_3$, generate distinct ideals in $\mathcal{O}$.
Proposition~\ref{prop: ideals} gives
\[
J(\varphi^{r+jm},\varphi^m) \mathcal{O}
= \prod_{k\in S_j} \sigma_k^{-1}(P),
\quad
S_j:=\{k \in \mathbb{Z}_{p-1}^\times : \pi(k (r+jm))+\pi(k m)<p-1\}
\qquad
\forall j \in \mathbb{Z}_3,
\]
so it suffices to show $S_0,S_1,S_2$ are distinct.
(Here we use the fact that $\mathcal{O}$ is a Dedekind domain and the factors in~\eqref{eq: factor pO} are distinct prime ideals.)

First, we distinguish $S_2$ from $S_0$ and $S_1$ by showing $1\notin S_2$ while $1\in S_0\cap S_1$.
Indeed, recalling that $1 \leq r < m:=\tfrac{p-1}{3}$, $r \neq \tfrac{m}{2}$, we have 
\[
\pi(r+jm)+\pi(m)=r+jm+m
=r+(j+1)m
<(j+2)m
\qquad
\forall j\in \{0,1,2\}.
\]
When $j\in \{0,1\}$ we have $(j+2)m\leq 3m=p-1$, so $1 \in S_j$; when $j=2$ we have $r+(j+1)m=r+(p-1)>p-1$, so $1 \notin S_j$.
Thus, $S_2$ differs from each of $S_0$ and $S_1$.

It remains to distinguish $S_0$ from $S_1$ by producing a member of $\mathbb{Z}_{p-1}^\times$ that belongs to one but not the other.
To achieve this, first consider the algebraic integer $a:=1+\omega_{p-1}^{r+m} \in \mathbb{Q}(\omega_{p-1})$.
Here, $a \neq 0$ by our assumption $r \neq \tfrac{m}{2}$: since $0 < r+m<3m=p-1$ we would have $a=0$ only if $r+m=\tfrac{p-1}{2}=\tfrac{3m}{2}$.
Then $b:=\prod_{k\in \mathbb{Z}_{p-1}^\times} \sigma_k(a) \neq 0$ is fixed by every field automorphism of $\mathbb{Q}(\omega_{p-1})$, so it is rational; since $b$ is also an algebraic integer, it must be an integer.
Consequently,
\[
1 \leq |b|^2
=\prod_{k\in \mathbb{Z}_{p-1}^\times} |\sigma_k(a)|^2
=\prod_{k\in \mathbb{Z}_{p-1}^\times} \Bigl|1 + \omega_{p-1}^{k(r+m)} \Bigr|^2,
\]
and there exists $k \in \mathbb{Z}_{p-1}^\times$ for which $\Bigl|1 + \omega_{p-1}^{k(r+m)} \Bigr|^2 \geq 1$.
Furthermore, we may take a complex conjugate to replace $k$ with $-k$ if necessary and assume that $k = 1 \bmod 3$, where $k \neq 0 \bmod 3$ since it is coprime with $p-1=0 \bmod 6$.

In the remainder of the proof, we show that this choice of $k$ belongs to exactly one of $S_0$ and $S_1$ by showing that exactly one of the two numbers
\[
\pi(kr)+\pi(km),
\qquad
\pi(kr+km)+\pi(km)
\]
is less than $p-1$.
Here, $\pi(km)=m$ since $km-m=(k-1)m$ is divisible by $3m=p-1$ by our assumption $k=1 \bmod 3$.
Now abbreviate $\ell := \pi(kr)$, so that $kr+km=\ell+m \bmod (p-1)$; then our task is to show exactly one of
\[
\pi(kr)+\pi(km)
=\ell+m,
\qquad
\pi(kr+km)+\pi(km)
=\pi(\ell+m)+m
\]
is less than $p-1=3m$, or equivalently, exactly one of $\ell, \pi(\ell+m)$ is less than $2m$.

We claim that one of two things happens: $m < \ell < 2m$, or $2m < \ell < 3m$.
To see this, first use the fact that $kr+km = \ell+m \bmod (p-1)$ to obtain
\[
1 \leq \Bigl|1 + \omega_{p-1}^{k(r+m)} \Bigr|^2
= \Bigl|1 + \omega_{p-1}^{\ell+m} \Bigr|^2
= 2+2\cos\Bigl( \tfrac{2\pi(\ell+m)}{3m}\Bigr),
\]
that is, $\cos\bigl( \tfrac{2\pi}{3}\cdot \tfrac{\ell+m}{m} \bigr) \geq -\tfrac{1}{2}$.
Consequently, $\tfrac{\ell+m}{m} \notin (1,2)$, that is, $\ell \notin (0,m)$.
To rule out the endpoints, assume for the sake of contradiction that $\ell\in\{0,m,2m\}$.
Then $m$ divides $rk$ since $3m$ divides $rk - \pi(rk)=rk-\ell$.
Here, $k\in \mathbb{Z}_{3m}^\times$ is coprime with $m$, so $m$ divides $r$, contrary to the fact that $0 < r < m$.
Consequently, $\ell \notin \{0,m,2m\}$; this proves the claim since $\ell = \pi(kr) \in [0,3m)$.

In the case where $m < \ell < 2m$, we have $2m <\ell+m<3m$, so that $\pi(\ell+m)=\ell+m$ and
\[
\ell < 2m < \pi(\ell+m).
\]
In the case where $2m < \ell < 3m$, we have $3m < \ell+m<4m$, so that $\pi(\ell+m)=\ell-2m<m$ and
\[
\pi(\ell+m)<2m<\ell.
\]
This completes the proof.
\end{proof}

\section{Discussion}
\label{sec: discussion}

We end with some discussion.
In this paper, we gave motivation for studying conference signals, described emergent features of a numerical search, gave nonconstructive proofs of widespread existence, and detailed two infinite families of explicit conference signals, one of composite order and the other of prime order.
Our explicit constructions made progress on a question that continues to drive us crazy: \textit{What are all these conference signals?}

\subsection{Future directions}

Eventually, we want to understand more constructions, as in the (ambitious) problem below.

\begin{problem}
\label{prob: all conference signals}
Find all conference signals on finite abelian groups, make sense of their constructions, and present intuitive explanations that others can understand.
\end{problem}

Problem~\ref{prob: emergent features} suggests some directions toward Problem~\ref{prob: all conference signals}; more follow.

\begin{problem}
For each item below, find the answer, understand the proof, and present an intuitive explanation that others can understand.
\begin{itemize}
\item[(a)]
For which $n$, $m\in \mathbb{Z}_n^\times$, and $\lambda \in \mathbb{C}$ does there exist a conference signal in the $\lambda$-eigenspace of $D_m$?
\item[(b)]
For which $n$ does $\mathbb{Z}_n$ admit an even conference signal?
\item[(c)]
For which $n$ does there exist a conference signal in the $1$-eigenspace of the DFT?
\item[(d)]
For which circulant unitaries $C_u$ does there exist a vector $x$ that is flat, not a multiple of a Fourier mode, and satisfies $u*x=\lambda x^*$ for some $\lambda\in \mathbb{C}$?
\item[(e)]
For every positive existence result above and in Problem~\ref{prob: emergent features}, explicitly construct all examples.
\end{itemize}
\end{problem}

\subsection{Similarities with SICs and Zauner's conjecture}

We feel another point deserves mention.
Throughout the course of this project, the authors have been struck by similarities between conference signals and SICs, where a \textbf{SIC} describes $d^2$ equiangular lines in $\mathbb{C}^d$.
For every~$d$, \textbf{Zauner's conjecture} predicts the existence of a SIC arising from a special vector in a certain eigenspace of the \textit{Fourier chirp} operator~\cite{Zauner:11}.
Readers familiar with this problem might appreciate the following commonalities between SICs and conference signals:
\begin{itemize}
\item[(i)]
conjectured existence at every point in an arithmetic family,
\item[(ii)]
conjectured existence in a specific eigenspace of a Fourier-like operator,
\item[(iii)]
seemingly isolated numerical solutions that are extremely difficult to reverse engineer,
\item[(iv)]
solutions with algebraic but non-cyclotomic entries,
\item[(v)]
unexpected interest by number theorists in a problem from harmonic analysis.
\end{itemize}
Furthermore, any odd conference signal of order~$n \geq 3$ produces a complex symplectic ETF consisting of~$n$ vectors in dimension~$n-1$, while the existence of a real symplectic ETF with those parameters for $n = 3 \bmod 4$ is described as an analogue of Zauner's conjecture in~\cite{FallonLAA:26}.

Despite these similarities, we have no hard evidence of a connection between the two areas, and we do not know any way to (say) turn a SIC into a conference signal, or vice versa.
Nevertheless, the analogy might be worth investigating further, especially since a solution of Zauner's conjecture was just announced~\cite{ApplebyFK:25,RadchenkoWh:26,RadchenkoW:26}.

\section{Use and non-use of AI}
\label{sec: AI}

ChatGPT made contributions for which any human would be a co-author.
Specifically, it: found Theorem~\ref{thm: n = 3p}; found initial proofs of preliminary versions of Theorems~\ref{thm: odd coneig existence}, \ref{thm: n = 3p}, \ref{thm: 3x3 circulant unitaries}, and Lemma~\ref{lem: irrationality and distinctness}; helped navigate the literature, drawing our attention to~\cite{Benoist:24} and the line of research on Cohn's problem; helped write and verify Mathematica code to implement the Newton--Kantorovich certifications in Example~\ref{ex: library}; made BibTeX entries; made the figures; wrote a preliminary draft of Section~\ref{sec: explicit prime} (see below); and suggested minor revisions for a near-final version of the manuscript.

The human authors did everything else.
In particular, we: initiated the project, which grew out of KF's PhD thesis~\cite{Fallon:26}; directed the research; interacted with ChatGPT much as we would with a human collaborator; digested, revised, and simplified the arguments supplied by ChatGPT; proved anything not explicitly identified above; and wrote the paper ourselves through Section~\ref{sec: explicit composite}.

For Section~\ref{sec: explicit prime}, JWI tried an experimental writing technique.
First, he gave ChatGPT his handwritten notes and a rough outline and had it adapt them into a draft of the section.
Then, he had it read his earlier papers, learn his style, and rewrite the draft to mimic his voice.
Next, he edited large-scale features by directing it to make specific changes to the organization, notation, and presentation until he was satisfied with everything but the language and level of detail.
Finally, he wrote the entire section again carefully himself, using ChatGPT's draft as a guide.
ChatGPT estimates that JWI replaced about 90\% of the language in this final (human) rewrite of Section~\ref{sec: explicit prime}.

The human authors bear full responsibility for any defects in this article, and we share credit for its merits with ChatGPT.

\section*{Further acknowledgments}
JWI was supported by NSF DMS 2220301.
The authors are pleased to thank Dustin Mixon (Ohio State) and Clay Shonkwiler (Colorado State) for insightful conversations.
In addition, JWI thanks Christina Iverson for a thoughtful conversation on the role of software assistance in crossword puzzle creation.
That conversation inspired his experimental approach to writing Section~\ref{sec: explicit prime}.

\bigskip

\bibliographystyle{abbrv}
\bibliography{refs}

\end{document}